\documentclass[letterpaper, 10 pt, conference]{ieeeconf}  % Comment this line out if you need a4paper
\IEEEoverridecommandlockouts                              % This command is only needed if 
\usepackage{amsmath, amssymb}
\usepackage{bm}
\usepackage[ruled,vlined,linesnumbered]{algorithm2e}
\usepackage{graphicx}
\usepackage{subcaption}
\usepackage{float}
\usepackage{tikz}
\usetikzlibrary{automata, positioning, arrows}
\usepackage{hyperref}
\usepackage{multirow}
\usepackage{makecell}
\usepackage{xcolor}
\usepackage{etoolbox}
\let\labelindent\relax
\usepackage[shortlabels]{enumitem}
\usepackage{mathrsfs}
\usepackage{booktabs}

\newtheorem{proposition}{\bfseries Proposition}
\newtheorem{example}{\bfseries Example}

\newtheorem{theorem}{\bfseries Theorem}

\newtheorem{problem}{\bfseries Problem}

\newcommand{\R}{\mathbb{R}}
\newcommand{\IR}{\mathbb{IR}}
\newcommand{\N}{\mathbb{N}}
\newcommand{\cS}{\mathcal{S}}
\newcommand{\cX}{\mathcal{X}}
\newcommand{\cU}{\mathcal{U}}

\newcommand{\cC}{\mathbb{C}}
\newcommand{\cO}{\mathbb{O}}
\newcommand{\Int}{\mathrm{int}}
\newcommand{\bigtimes}{\mathop{\raisebox{-.5ex}{\scalebox{1.6}{$\times$}}}\displaylimits} %  Cartesian Product

\makeatletter
\patchcmd{\@makecaption}{\scshape}{}{}{}
\patchcmd{\@algocf@start}{-1.5em}{0pt}{}{}
\makeatother

\graphicspath{{./figures/}}
\allowdisplaybreaks[4]

\newif\ifdraft
\drafttrue
\usepackage{todonotes}
\title{\LARGE \bf Computing the Maximal Controlled Invariant Set for Neural Network Control Systems}
\author{Tianxiao Ye, Hang Zhang, and Xiangru Xu
\thanks{T. Ye, H. Zhang, and X. Xu are with the Department of Mechanical Engineering, University of Wisconsin-Madison, Madison, WI, USA. Email: 
{\tt\small \{tye46,hang.zhang,xiangru.xu\}@wisc.edu}.}}

\begin{document}

\maketitle

% \begin{abstract}
% This paper considers the computation of controlled invariant sets (CISs) for discrete-time neural network control systems (NNCSs). An optimization-free method is proposed that computes certified inner and outer approximations of the maximal CISs. An inclusion function whose bounds depend affinely on the control input is constructed for NNCS dynamics, and a hyperplane arrangement induces a finite partition of the control space such that feasibility is uniform on each face, reducing the infinite search over admissible controls to the finite evaluation of a single representative per face. These independent evaluations are naturally parallelizable on GPUs. The effectiveness of the proposed approach is demonstrated through two numerical examples.
% \end{abstract}
\begin{abstract}
This article studies the computation of the maximal controlled invariant set (MCIS) for  neural network control systems (NNCSs) with a nominal plant model and a neural network residual model. An interval-based method is developed to construct a control-affine inclusion function for the NNCS dynamics and to exploit the induced hyperplane arrangement in the input space, which  reduces the verification of controlled invariance from an infinite control search to a finite set of representative evaluations. 
Based on this finite characterization, verification-guided algorithms  iteratively construct inner and outer approximations of the the MCIS, which are naturally parallelizable. Numerical examples demonstrate the effectiveness of the proposed approach.
% This article studies the computation of the maximal controlled invariant set (MCIS) for  neural network control systems (NNCSs) with a known plant model and a neural network residual. An efficient and optimization-free method is proposed to verify the CIS condition. The method constructs a control-affine inclusion function of NNCS dynamics and exploits the induced hyperplane arrangement to partition the control space into finitely many subsets. Therefore, it reduces the CIS verifications from the infinite control search to a finite number of representative evaluations. 
% Based on this finite characterization, verification-guided algorithms are developed for iteratively constructing the MCIS inner and outer approximations that are naturally parallelizable. Numerical examples demonstrate the effectiveness of the approach.
\end{abstract}

\section{Introduction}
\label{sec:intro}
With the extensive adoption of neural networks (NNs), an increasing number of dynamical systems incorporate NN components to universally approximate complex nonlinear behaviors that are difficult or expensive to derive from first principles~\cite{salzmann2023real,schwan2023stability,fabiani2022reliably,fazlyab2020safety}. This trend leads to the rapid development of neural network control systems (NNCSs).
Despite their empirical success, providing formal guarantees for NNCSs remains a significant challenge due to the highly nonlinear and compositional structure of NNs.
% Safety is a primary concern in control systems, particularly in applications where constraint violations can result in catastrophic failures.

Controlled invariant sets (CISs) play a fundamental role in formally analyzing control systems because they characterize the states from which state and input constraints can be satisfied for all future time under admissible control inputs.
Hence, CISs serve as a key tool for designing safety filters and guaranteeing recursive feasibility in model predictive control~\cite{ames2016control,xu2015robustness,rawlings2020model,mayne2000constrained}. 
For linear systems, efficient methods to compute CISs are proposed based on the Riccati-type recursive matrix equation~\cite{bertsekas1972infinite}, polyhedral backward reachability iterations~\cite{blanchini1999set}, and Lagrangian methods~\cite{maidens2013lagrangian}. 
For nonlinear systems, the problem becomes more challenging due to the loss of convexity and a tractable analytical structure. Existing approaches typically rely on convex approximations, interval analysis, or reachable-set over/under-approximations, often at the cost of conservatism or additional assumptions~\cite{fiacchini2010computation}, \cite{monnet2016computing}, \cite{brown2023computing}. However, these methods are not directly applicable to NNCSs, and synthesizing CISs of NNCSs is still an open problem. A recent work~\cite{li2025control} addresses this for rectified linear unit (ReLU)-activated multilayer perceptrons by combining state-space quantization with mixed-integer linear programming (MILP). Nevertheless, the proposed approach is limited to stand-alone NNs representing system dynamics; moreover, its reliance on MILP hinders scalability to larger networks and restricts applicability to piecewise-linear activation functions.

This paper addresses the inner and outer approximations of the maximal controlled invariant set (MCIS) for NNCSs whose dynamics consist of a nominal plant model and an NN residual model.
Our main contributions are twofold: 
(i) We develop an efficient, optimization-free method for the verification of a relaxed controlled invariance condition, based on the function inclusion and hyperplane arrangement techniques. Specifically, we construct a control-affine inclusion function of the NNCS dynamics and show that the induced hyperplane arrangement for this inclusion function partitions the admissible control space into finite subsets, within each of which the verification result is invariant. This property converts the infinite existential search over admissible controls to the evaluation of a single representative point per subset.  
% (ii) We show that for each state computed in the CIS, the same finite partition yields a box-wise certified feasible control set represented as a finite union of convex polytopes, thereby providing a certified under-approximation of the regulation map.
(ii) We develop verification-guided synthesis algorithms to compute the inner and outer approximation of the MCIS represented as the union of intervals. The proposed algorithms are inherently parallelizable, making them well-suited for high-performance computing environments, such as GPU-based implementations.
% Our main contributions are threefold: (i) we develop a certified method for computing a three-class paving of the safe set $\cX$ into certified three subpavings: an inside subpaving $\cC$ certified to be controlled invariant, an outside subpaving $\cO$ certified to lie outside the maximal CIS, and an unclassified subpaving~$\mathbb{U}$. (ii) we introduce a boxwise invariance operator and realize it without optimization by constructing a control-affine inclusion function, showing that the induced hyperplane arrangement gives rise to a finite partition of the admissible control space over which the feasibility indicator remains constant, and thereby reducing the infinite existential search over admissible controls to the evaluation of one representative point per partition. (iii) we show that for each certified box, the same finite partition yields a boxwise certified feasible control set represented as a finite union of convex polytopes, thereby providing a certified under-approximation of the regulation map.

The remainder of this article is organized as follows. Section~\ref{sec:pre} introduces preliminaries and the problem statement; Section~\ref{sec:suffCIS} introduces a relaxed CIS condition via function inclusion; Section~\ref{ssec:affine} develops a control-affine enclosure of the NNCS dynamics; Section~\ref{ssec:critical_ctrl} shows how the relaxed CIS test can be reduced to a finite search over representative controls via hyperplane arrangements;  Section~\ref{ssec:complete} presents the verification-guided algorithms for constructing the inner and outer approximations of the MCIS;  Section~\ref{sec:experiments} demonstrates the effectiveness of the proposed approach through numerical examples; finally, Section~\ref{sec:concl} concludes the article.

\section{Preliminaries \& Problem Statement}
\label{sec:pre}

\subsection{Preliminaries}
% $\R^n$ the $n$-dimensional Euclidean space, 
The set of real numbers and natural numbers are denoted by $\R$ and $\N$, respectively. For a given set $A \subseteq \R^n$, $\Int(A)$, $\partial A$,
$\overline{A}$, and $2^A$ denote its interior, boundary, closure,
and power set, respectively. For a given set $A$ that is finite, $|A|$ denotes its
cardinality. The symbol $B(x,r)$ denotes the open $\ell_\infty$-ball of radius $r$ centered at $x$ in $\R^n$, i.e., 
$B(x,r)=\left\{z \in \mathbb{R}^{n} \;|\;\|z-  x\|_{\infty} < r\right\}$, where $\|\cdot\|_\infty$ is the infinity norm.

\subsubsection{Interval Analysis}\label{ssec:interval}

We introduce notations and definitions from interval analysis that will be used in the sequel~\cite{moore2009introduction,jaulin2001interval}. 
The set of all intervals in $\R$ is denoted by $\IR$. A closed interval in $\IR$ is denoted by $[\bm x] = [\underline{x}, \overline{x}]$ and an open interval in $\IR$ is denoted by $(\bm x) = (\underline{x}, \overline{x})$, where $\underline{x}$ and $\overline{x}$ are the lower and upper bounds, respectively, satisfying $\underline{x}\leq \overline{x}$ elementwise.
% A box is defined as the Cartesian product of intervals:  $[x] = [x_1] \times \cdots \times [x_n] \in \IR^n$.
Given $n$ intervals $[\bm x_1],\dots,[\bm x_n]\in\IR$, their Cartesian product is
denoted by $\bigtimes_{i=1}^{n}[\bm x_i] \triangleq [\bm x_1]\times\cdots\times[\bm x_n]$.
A box $[{\bm x}]\in\IR^n$ is defined as such a Cartesian product, i.e., 
$[\bm x] = \bigtimes_{i=1}^{n}[\bm x_i]$.
The width and midpoint of a box are defined as
$w([\bm x]) \triangleq \|\overline{\bm x} - \underline{\bm x}\|_\infty, \mathrm{mid}([\bm x]) \triangleq \tfrac{1}{2}(\underline{\bm x} + \overline{\bm x})$, respectively. 
For $y\in\mathbb{R}^n$, we denote by $\{y\}$ the degenerate box $[y, y]\in\mathbb{IR}^n$ and a box is called non-degenerate if it has nonempty interior.
% For a set $\mathcal{S}\subset\mathbb{R}^n$, $\Box\mathcal{S}$ denotes its interval hull, i.e., the smallest box that contains  $\mathcal{S}$.

%\begin{definition} \cite{jaulin2001interval}
Given a function $f: \R^n \to \R^m$, an interval function $[f]: \mathbb{IR}^n \to \mathbb{IR}^m$ is called an \emph{inclusion function}  for $f$ if $f([\bm x]) \subseteq [f]([\bm x])$ for any $[\bm x]\in\IR^n$ \cite[Section 2.4]{jaulin2001interval}. 
%The inclusion function $[f]$ is said to be \emph{convergent} if, for any sequence of boxes $[x](k)$, $\lim_{k\rightarrow\infty}w([x](k))=0\implies \lim_{k\rightarrow\infty}w([f]([x](k))) = 0$.
%\end{definition}

% Given a map $g: \R^n \to \R^m$, its exact image over a box $[x]\in\mathbb{IR}^n$ is $g([x]) := \{g(x) : x \in [x]\}$. 
%  the interval function $[g]: \mathbb{IR}^n \to \mathbb{IR}^m$ is an inclusion function  for $g$ if $g([x]) \subseteq [g]([x])$ for any $[x]$. The inclusion function is said to be \emph{convergent} if $w([g]([x])) \to 0$ as $w([x]) \to 0$.

%[subpaving and paving
%\begin{definition}\cite{chabert2009contractor}\label{def:subpaving}
% A \emph{subpaving} of a box $[x]$ is a finite set $\cC \subset \IR^n$ of boxes $[{b}] \subseteq [x]$ with pairwise disjoint interiors and we associate to $\cC$ the point-set union $\cC \triangleq \bigcup_{[{b}] \in \cC} [{b}].$
A \emph{subpaving} of a box 
$[\bm x]\in\IR^n$ is a finite set of boxes $\{[\bm b_1], \ldots, [\bm b_N]\} \subseteq \IR^n$
such that $[{\bm b}_i] \subseteq [\bm x]$ for all $i$ and
$\Int([{\bm b}_i]) \cap \Int([{\bm b}_j]) = \emptyset$
for any $i \neq j$. Besides viewing a subpaving as the set of boxes $\{[\bm b_1], \ldots, [\bm b_N]\}$, it can  also be viewed as the union $[\bm b_1]\cup\dots\cup [\bm b_N]$, depending on the context. Hence, a subpaving can be interpreted either as a discrete subset of $\IR^n$ or as a compact subset of $\R^n$  \cite[Section 1.5]{chabert2009contractor}. A \emph{paving} of a box $[\bm x]\in\IR^n$ is a collection of subpavings
$\mathbb{K}_1, \ldots, \mathbb{K}_N$ such that $[\bm x] =  \bigcup_{i=1}^{N} \bigcup_{[{\bm b}] \in \mathbb{K}_{i}} [{\bm b}]$.

\subsubsection{Controlled Invariant Set}
Consider a discrete-time control system $x^+ = f(x,{u})$ where the state $x\in\cX\subseteq \R^n$ and input $u\in\cU\subseteq \R^m$. A set $\cS \subseteq \cX$ is called a \emph{controlled invariant set} (CIS) if, for every $x \in \cS$, there exists $u \in \cU$ such that $f(x, u) \in \cS$  \cite{blanchini1999set,aubin2011viability}. The \emph{maximal controlled invariant set} (MCIS) is the largest CIS, defined as 
% \begingroup
% \setlength{\abovedisplayskip}{3pt}  % default is ~12pt
% \setlength{\belowdisplayskip}{3pt}
$$
\cS^* = \bigcup \{\cS \subseteq \cX : \cS \text{ is a CIS}\}.
$$
%\endgroup
%with state constraint set $\cX$ and control constraint set $\cU$.
The \emph{regulation map} associated with a CIS  $\cS\subseteq\cX$ is the set-valued map $R_{\cS}:\cX\to 2^{\cU}$ defined by $  R_{\cS}(x) := \{u\in\cU : f(x,u)\in\cS\}.$
%[Controlled Invariant Set]
% \begin{definition}\label{def:CIS}
% A set $\cS \subseteq \cX$ is called a \emph{controlled invariant set} (CIS) of the discrete-time $x^+ = f(x,u)$ if, for every $x \in \cS$, there exists $u \in \cU$ such that $f(x, u) \in \cS$. The \emph{maximal controlled invariant set} (MCIS) is the largest CIS containing all CISs, defined as
% $\cS^* = \bigcup \{\cS \subseteq \cX : \cS \text{ is a CIS}\}$. 
% \end{definition}
%The maximal CIS coincides with the viability kernel of $\cX$ .

The \emph{one-step predecessor} of a set $\cS\subseteq\cX$ is defined as
% \begingroup
% \setlength{\abovedisplayskip}{3pt}  % default is ~12pt
% \setlength{\belowdisplayskip}{3pt}
\begin{equation*}%\label{eq:Pre}
  \mathrm{Pre}(\cS)
  :=
  \{x\in\cX : \exists u\in\cU,\; f(x,u)\in\cS\}.
\end{equation*}
%\endgroup
Clearly, a set $\cS\subseteq\cX$ is a CIS if and only if
$$
\cS \subseteq \mathrm{Pre}(\cS).
$$ 
%[Regulation map]
%\begin{definition}\label{def:regulation}

% \begin{equation*}%\label{eq:regulation}
%   R_{\cS}(x) := \{u\in\cU : f(x,u)\in\cS\}.
% \end{equation*}
%\end{definition}

% The CIS and the regulation map are closely related to concepts in viability theory; see~\cite{aubin2011viability} for more details.

\subsubsection{Hyperplane Arrangements}
% \subsection{Controlled Invariance}
% \label{ssec:ci}
We now introduce several definitions from the theory of hyperplane arrangements, following \cite{stanley2007introduction,orlik2013arrangements},  that will be used in Section~\ref{ssec:critical_ctrl}. 
An \emph{affine hyperplane} in $\R^n$ is defined as $\{v\in\R^n\;|\; \alpha^\top v=\beta\}$ where $\alpha\in\R^n$ is a nonzero vector and $\beta\in \R$.
%, and $\alpha \cdot v$ is the dot product. 
A (finite) \emph{hyperplane arrangement} in $\R^n$ is a finite set of affine hyperplanes in $\R^n$, denoted as $\mathcal{A}\triangleq\{\mathcal{H}_1,\mathcal{H}_2,\cdots,\mathcal{H}_N\}$, where $\mathcal{H}_i=\{v\in\R^n\;|\; \alpha_i^\top v=\beta_i\}$, $i=1,2,\dots,N$.  
%$\mathcal{A}\triangleq\cup_{i=1}^{N} \{\{v\in\R^n\;|\; \alpha_i \cdot v=\beta_i\}\}$. 
%Now let $\mathbb{R}^d$ to be a given finite-dimensional vector space.
%\begin{definition}\label{def:sign_pattern}
The \emph{sign pattern} of a vector $x\in\R^n$ with respect to $\mathcal{A}$ is defined as $s(x)\triangleq\left(\operatorname{sgn}(\alpha_1^\top x-\beta_1),\dots,\operatorname{sgn}(\alpha_N^\top x-\beta_N)\right)\in \{-1,0,+1\}^N$. The (open) \emph{face} of $\mathcal{A}$ is a maximal connected set of points sharing the same sign pattern. The collection of all faces forms a partition of $\R^n$, i.e., every point in $\R^n$ belongs to exactly one face. Faces may have any dimension from $0$ to $n$. The full-dimensional faces are open convex polyhedra and are called the \emph{regions} of the arrangement.
% \begin{align*}
%   s(x) := \bigl(\operatorname{sgn}(\ell_{\mathcal{H}}(x))
%   \bigr)_{\mathcal{H}\in\mathcal{A}} \in \{-1,0,+1\}^{|\mathcal{A}|}.
% \end{align*}
% where $\operatorname{sgn}(t)$ equals $-1$, $0$, or $+1$ according to whether $t<0$, $t=0$, or $t>0$.
%\end{definition}

%\begin{definition}\cite[Section~0]{zaslavsky1975facing}\label{def:face_region}
% A \emph{face} of a finite hyperplane arrangement $\mathcal{A}$ in $\R^d$ is a maximal connected set of points sharing the same sign pattern. The collection of all faces forms a partition of $\R^d$, i.e., every point in $\R^d$ belongs to exactly one face. Faces may have any dimension from $0$ to $d$. The full-dimensional faces are open convex polyhedra~\cite[Section~1.1]{stanley2007introduction} and are called the \emph{regions} of the arrangement.
%\end{definition}

%The maximal CIS coincides with the viability kernel of $\cX$, while the regulation map defined above generalizes the classical regulation map of the viability kernel  ~\cite{aubin2011viability}. 

%This definition generalizes the classical regulation map of the viability kernel~\cite[Definition~9]{aubin2011viability}, which corresponds to the special case $\cS=\cS^*$.

\subsection{Problem Statement}
\label{ssec:problem}

%It is assumed that we are given a time-invariant discrete-time dynamic system
In this work, we consider the following discrete-time control system:
\begingroup
\setlength{\abovedisplayskip}{3pt}  % default is ~12pt
\setlength{\belowdisplayskip}{3pt}
\begin{equation}\label{eq:system}
  x_{k+1} = \underbrace{f_0(x_k, u_k)+f_{\mathrm{NN}}(x_k, u_k)}_{f(x_k,u_k)} %\quad k = 0, 1, 2, \ldots
\end{equation}
\endgroup
where $x_k \in \cX\subseteq\R^n$ is the state, $u_k \in \cU\subseteq \R^m$ is the control input, $f_0:\R^n\times\R^m\to\R^n$ represents the nominal dynamics, and $f_{\mathrm{NN}}:\R^n\times\R^m\to\R^n$ is a function parameterized as a feed-forward neural network (FFNN) that captures the residual dynamics. For clarity of presentation, we assume that both $\cU$ and $\cX$ are boxes; however, the proposed method extends naturally to the case where $\cX$ and $\cU$ are finite unions of boxes. Assume that $f_0$ is  twice continuously differentiable on an open set containing $\cX \times \cU$, with bounded second-order partial derivatives. Assume that $f_{\mathrm{NN}}$ has $\ell-1$ hidden layers; let $W^{(i)}\in\R^{n_i\times n_{i-1}}$ and $b^{(i)}\in\R^{n_i}$ denote the weight matrix and bias vector of the $i$-th layer, respectively, where $n_i$ is the number of neurons in the $i$-th layer, with $n_0 = n+m$ and $n_{\ell}=n$;  denote by $z^{(i)}$ the output of the neurons in the $i$-th layer. Then, we have  
%Writing $f := f_0 + f_{\mathrm{NN}}$, we shall refer to $f_0$ as the \emph{nominal dynamics} and to $f_{\mathrm{NN}}$ as the \emph{residual dynamics}.
% \XX{can we remove the superscript f in these n?}\TY{Yes. Done.}\XX{keep simplifying the notations}
% \begin{align}
%   z^{(0)} &= \begin{bmatrix} x_k \\ u_k \end{bmatrix} \in \R^{n_0}, \label{eq:f_input} \\
%   \hat{z}^{(i)} &= W^{(i)} z^{(i-1)} + b^{(i)},\, i=1,\ldots,\ell_f-1, \label{eq:f_linear} \\
%   z^{(i)} &= \sigma\!\bigl(\hat{z}^{(i)}\bigr), \, i=1,\ldots,\ell_f-1, \label{eq:f_act} \\
%   f_{\mathrm{NN}}(x_k,u_k) &= W^{(\ell_f)} z^{(\ell_f-1)} + b^{(\ell_f)} \in \R^n, \label{eq:f_output}
% \end{align}
\begin{align*}
z^{(0)} &= \begin{bmatrix} x_k \\ u_k \end{bmatrix} \in \R^{n+m}, \\
z^{(i)} &= \sigma\!\bigl(W^{(i)} z^{(i-1)} + b^{(i)}\bigr)\in \R^{n_i}, \, i=1,\ldots,\ell-1, \label{eq:f_act} \\
f_{\mathrm{NN}}(x_k,u_k) &= W^{(\ell)} z^{(\ell-1)} + b^{(\ell)} \in \R^n,
\end{align*}
where $\sigma(\cdot)$ denotes the activation function (such as ReLU and $\tanh$) applied element-wise.

We will address the problem stated as follows.
\begin{problem}\label{prob:main}
Given the NNCS \eqref{eq:system}, with $x_k \in \cX\subseteq\R^n$ and $u_k \in \cU\subseteq \R^m$, where both $\cU$ and $\cX$ are boxes, compute two subpavings $\underline{\cS}^*$ and $\overline{\cS}^*$ as the inner and outer approximations of the MCIS, respectively, such that 
$$
\underline{\cS}^*\subseteq\cS^*\subseteq\overline{\cS}^*
$$ 
with $\underline{\cS}^*$ being a CIS, while making the approximations as tight as possible.
\end{problem}

% \begin{problem}
% \label{prob:main}
% Given $f_0$, $f_{\mathrm{NN}}$, $\mathcal{X}$, and $\mathcal{U}$, compute:
% \begin{enumerate}[label=(\roman*), nosep]
    % \item a paving of $\cX$ into three subpavings $\cC$, $\cO$, and $\mathbb{U}$ such that
    % \[
    % \begin{aligned}
    % \cC \cup \cO \cup \psu{\mathbb{U}} &= \cX,\\
    % \cC &\subseteq \cS^*,\\
    % \cO \cap \cS^* &= \emptyset,\\
    % \Int(\cC) \cap \Int(\cO) &= \emptyset,\\
    % \Int(\cC) \cap \Int(\psu{\mathbb{U}}) &= \emptyset,\\
    % \Int(\cO) \cap \Int(\psu{\mathbb{U}}) &= \emptyset,
    % \end{aligned}
    % \]
    % where $\cC$ is a certified CIS of system~\eqref{eq:system}.
    
To solve this problem, we construct a paving of $\cX$ of the form  $\cX=\mathbb{C} \cup \cO \cup \mathbb{U}$ such that $\mathbb{C} \subseteq \cS^*$ and  
$\cO \subseteq \mathcal{X}\setminus \cS^*$. Here, $\mathbb{C}$ is a CIS, $\cO$ contains states that lie outside the MCIS, and $\mathbb{U}$ consists of states that remain undetermined. Therefore, Problem \ref{prob:main} can be addressed by setting $\underline{\cS}^*=\mathbb{C}$ and $\overline{\cS}^*=\mathbb{C}\cup \mathbb{U}$ while making $\mathbb{U}$ as small as possible.  The detailed algorithms for constructing the subpavings $\mathbb{C}$, $\cO$, and  $\mathbb{U}$ are presented in Section \ref{ssec:complete}, following the introduction of a CIS sufficient condition via function inclusion in Section \ref{sec:suffCIS}, the control-affine enclosure of the NNCS dynamics in Section \ref{ssec:affine} and the finite search for a feasible control via hyperplane arrangements in Section \ref{ssec:critical_ctrl}.

\section{Relaxed CIS Condition via Function Inclusion}\label{sec:suffCIS}

Recall that a set $\cS$ is a CIS if and only if $\cS \subseteq \mathrm{Pre}(\cS)$. For a subpaving $\mathbb{S}$, we relax this condition by using an inclusion function $[f]$ for $f$ and verifying each box $[\bm b] \in \mathbb{S}$ individually: for each $[\bm b]$, find $u \in \cU$ such that 
$$
[f]([\bm b],u) \subseteq \mathbb{S}
$$
where $[f]([\bm b], u) \supseteq \{f(x, u) : x \in [\bm b]\}$ is the inclusion function evaluated on the box $[\bm b]$ with control $u$. The following proposition establishes the soundness and completeness of this relaxed condition.
%We now present a proposition that reveals the conservatism of this sufficient condition.

%However, computing $\mathrm{Pre}(\cS)$ requires searching over an infinite input set $\cU$ to find a suitable control $u$ for every $x \in \cS$, which is computationally prohibitive, especially for the NNCS \eqref{eq:system}. To address Problem \ref{prob:main}, we relax the CIS condition to a box-based sufficient condition by using inclusion functions. 
%: for each $[\bm b] \in \mathbb{S}$, find $u \in \cU$ such that $[f]([\bm b],u) \subseteq \mathbb{S}$.

% Furthermore, recall from Section~\ref{ssec:affine} that the CIS condition $\mathbb{S} \subseteq \mathrm{Pre}(\mathbb{S})$ requires the existence of $u \in \cU$ with $f(x,u) \in \mathbb{S}$ for every $x \in \mathbb{S}$. 
% We relax this to a box-level sufficient condition: for each $[\bm b] \in \mathbb{S}$, find $u \in \cU$ such that $[f]([\bm b],u) \subseteq \mathbb{S}$.
% We now present a proposition that reveals the relationship between this sufficient condition and the exact condition.

\begin{proposition}\label{prop:sound_complete}
Let $\mathbb{S}$ be a subpaving of $\cX$ and $[f]$ an inclusion function for $f$. Then the following two statements hold.
\begin{enumerate}[label=(\roman*), nosep]
  \item\label{item:soundness}
  \textbf{(Soundness.)}
  If there exists $u^* \in \cU$ such that $[f]([\bm b], u^*) \subseteq \mathbb{S}$, then $[\bm b] \subseteq \mathrm{Pre}(\mathbb{S})$.

  \item\label{item:completeness}
  \textbf{(Completeness.)}
  Suppose that $[f]$ is convergent, i.e., for any sequence of boxes $[\bm x](k)$, $\lim_{k\rightarrow\infty}w([\bm x](k))=0$ $\implies$ $\lim_{k\rightarrow\infty}w([f]([\bm x](k))) = 0$. 
  If $x \in \mathrm{Pre}(\mathbb{S})$ and there exists $u^* \in \cU$ with 
  $f(x, u^*) \in \Int(\mathbb{S})$, then there exists $\delta > 0$ such that, for all $[\bm b]$ containing $x$ with $w([\bm b]) < \delta$,  
  $[f]([\bm b], u^*) \subseteq \mathbb{S}$.
\end{enumerate}
\end{proposition}

\begin{proof}
\ref{item:soundness}.
Since $[f]$ is an inclusion function, it holds that $f([\bm b], u^*) \subseteq [f]([\bm b], u^*)$, and therefore, $f([\bm b], u^*) \subseteq \mathbb{S}$.
In particular, every $x \in [\bm b]$ satisfies $f(x, u^*) \in \mathbb{S}$, which means that $u^*$ witnesses the containment $[\bm b] \subseteq \mathrm{Pre}(\mathbb{S})$.

\ref{item:completeness}.
Since $f(x, u^*) \in \Int(\mathbb{S})$, there exists $r > 0$ such that $B(f(x, u^*), r) \subseteq \mathbb{S}$.
Because $[f]$ is convergent and $\{u^*\} = [u^*, u^*]$ is a degenerate box, it holds that $w([\bm b] \times \{u^*\}) = w([\bm b])$.
Since $[f]$ is convergent, there exists $\delta > 0$ such that $w([\bm b]) < \delta \Longrightarrow w\bigl([f]([\bm b], u^*)\bigr) < r.$
Now for any $x \in [\bm b]$ with $w([\bm b]) < \delta$, 
by the inclusion property, $f(x, u^*) \in [f]([\bm b], u^*)$.
For every $y \in [f]([\bm b], u^*)$, since both $y$ and $f(x, u^*)$ lie in the box $[f]([\bm b], u^*)$, we have $\|y - f(x, u^*)\|_\infty \leq w\bigl([f]([\bm b], u^*)\bigr) < r$.
Therefore, $[f]([\bm b], u^*) \subseteq B(f(x, u^*), r) \subseteq \mathbb{S}$.%, where $u^*$ witnesses the existential test in the invariance condition.
\end{proof}

% Soundness shows that the sufficient condition implies the original CIS condition, while completeness shows that, under mild conditions on $f$ and $\mathbb{S}$, the gap between them can be sufficiently small if the box width is sufficiently small.
% Proposition \ref{prop:sound_complete} shows that the sufficient condition implies the original CIS condition, and moreover, under mild conditions on $f$ and $\mathbb{S}$, the gap between them can be sufficiently small if the box width is sufficiently small.

Proposition \ref{prop:sound_complete} shows that, under mild conditions on $f$ and $\mathbb{S}$, the gap between the relaxed and the original inclusion conditions for CIS can be made arbitrarily small by refining the box widths. However, the relaxed condition still requires searching over an infinite input set $\cU$ to find a suitable control input, which is computationally prohibitive, especially for the NNCS \eqref{eq:system}. To address this issue, the next two sections develop methods that reduce this infinite search to a finite search over representative control inputs, by using a control-affine enclosure of NNCS dynamics and hyperplane arrangements in the input space.

%if the box width is sufficiently small.
% \XX{I removed the part for $f_0$ to appendix. Revise this section accordingly.}

% To realize the box-level invariance operator efficiently, we require an inclusion function for the dynamics in~\eqref{eq:system} that is computationally tractable. 
% % \HZ{What is an invariance operator? May need more explanation before the formal definition in Sec IV.}
% % Its quality directly affects the pruning power of the invariance test and, consequently, the overall runtime of the algorithm.
% Our goal in this subsection is to construct an inclusion function that preserves the explicit dependence on the control variable $u$.

% More precisely, for a fixed state box $[x]$, we seek a box-valued mapping of $u$ that encloses the image $f([x],u)$ and whose faces depend affinely on $u$.
% We refer to such a mapping as a \emph{control-affine enclosure}.
% This structure will later allow us to reduce the existential search over admissible controls to a finite feasibility check over representative control values.

%with $[x]\subseteq\mathcal{X}$ and $[u]\subseteq\mathcal{U}$

\section{Control-affine Enclosure of NNCS Dynamics}\label{ssec:affine}

% \XX{Note: I use bold symbol to represent intervals, as I find it weird to express $\forall x\in[x]$. Please change other parts accordingly, i.e., use  bold symbols to represent intervals.}
In this section, we derive affine functions of $u$ that bound $f(x,u)$ from above and below over arbitrarily given boxes in $\mathcal{X}\times \mathcal{U}$. 
Specifically, 
%given any box $[\bm x]\in\IR^n$ and any box $[\bm u]\in\IR^m$, 
given any $[\bm x]\subseteq\mathcal{X}$  and $[\bm u]\subseteq\mathcal{U}$,  we aim to determine $\underline{a}_i\in\R^m$, $\overline{a}_i\in\R^m$, $\underline{c}_i\in\R$, and $\overline{c}_i\in\R$, where $i=1,2,\dots,n$, such that 
\begingroup
\setlength{\abovedisplayskip}{3pt}  % default is ~12pt
\setlength{\belowdisplayskip}{3pt}
\begin{equation}\label{closuref} 
f(x,u) \in [f]_{[\bm x]}(u),\quad \forall x\in[\bm x],\;\forall u\in[\bm u],  
\end{equation}
\endgroup
where 
\begingroup
\setlength{\abovedisplayskip}{3pt}  % default is ~12pt
\setlength{\belowdisplayskip}{3pt}
\begin{equation}\label{eq:combined_img}
[f]_{[\bm x]}(u)
\triangleq
\bigtimes_{i=1}^{n}
\bigl[\underline{{a}}_i^{\top}u+\underline{c}_i,\;\; \overline{{a}}_i^{\top}u+\overline{c}_i\bigr].
\end{equation}
\endgroup
We refer to $[f]_{[\bm x]}(u)$ in \eqref{eq:combined_img} the \emph{control-affine enclosure} of $f$ over $[\bm x]$.
%,  $\underline{{a}}_j$ and $\overline{{a}}_j$ the lower and upper slope vectors, receptively,  and  $\underline{c}_j,\overline{c}_j$ the corresponding intercepts.
Since  $f=f_0+f_{\mathrm{NN}}$,  we construct separate control-affine enclosures for $f_0$ and $f_{\mathrm{NN}}$, and then combine them through summation.

\subsection{Control-affine Enclosure of \texorpdfstring{$f_0$}{f0}} 
Suppose that the nominal dynamics $f_0$ is written as $f_0=(f_{0,1},f_{0,2},\dots,f_{0,n})^\top$. We can apply the standard second-order Taylor expansion with an interval-bounded remainder to construct a control-affine enclosure for $f_{0,i}$ and, consequently, for $f_{0}$~\cite{moore2009introduction}. 

Let $x_c=\mathrm{mid}([\bm x])$ and  $u_c=\mathrm{mid}([\bm u])$ denote the  respective midpoints of $[\bm x]$ and $[\bm u]$. For any $i\in\{1,2,\dots,n\}$, we can expand $f_{0,i}$ around $(x_c,u_c)$ as 
$f_{0,i}(x,u)=f_{0,i}(x_c,u_c)+\nabla_{u} f_{0,i}(x_c,u_c)^\top (u-u_c)+\nabla_{x} f_{0,i}(x_c,u_c)^\top (x-x_c)+ r_i$. Here, $r_i$ is the Lagrange remainder term satisfying $|r_i|\leq\Delta_i\triangleq\tfrac{1}{2}{\rho}^\top\overline{H}_i {\rho}$ where $\rho = \tfrac{1}{2}(\overline{x}_1 - \underline{x}_1, \ldots, \overline{x}_n - \underline{x}_n, \overline{u}_1 - \underline{u}_1, \ldots, \overline{u}_m - \underline{u}_m)^\top \in \R_{\geq 0}^{n+m}$ is the vector of componentwise half-widths of $[\bm x] \times [\bm u]$, and $\overline{H}_i \in \R_{\geq 0}^{(n+m)\times(n+m)}$ is an elementwise upper bound on the absolute value of 
the Hessian of $f_{0,i}$ over $[\bm x] \times [\bm u]$. This bound is computed by evaluating each entry of $H_i$ over $[x] \times [u]$ via natural interval extension~\cite[Section~5.3]{moore2009introduction}. 
Furthermore, the linear function $\nabla_{x} f_{0,i}(x_c,u_c)^\top x$ can be lower and upper bounded as  $\underline{\theta}_i
\leq\nabla_{x} f_{0,i}(x_c,u_c)^\top x\leq \overline{\theta}_i$, where $\underline{\theta}_i=\min_{x\in[\bm x]}\nabla_{x} f_{0,i}(x_c,u_c)^\top x$ and  $\overline{\theta}_i=\max_{x\in[\bm x]}\nabla_{x} f_{0,i}(x_c,u_c)^\top x$. 

Setting ${\delta}_{i} =\nabla_{u} f_{0,i}(x_c,u_c)$,  $\underline{{\eta}}_{i}=f_{0,i}(x_c,u_c)-\nabla_{u} f_{0,i}(x_c,u_c)^\top u_c-\nabla_{x} f_{0,i}(x_c,u_c)^\top x_c+\underline{\theta}_i-\Delta_i$, and $\overline{{\eta}}_{i}=f_{0,i}(x_c,u_c)-\nabla_{u} f_{0,i}(x_c,u_c)^\top u_c-\nabla_{x} f_{0,i}(x_c,u_c)^\top x_c+\overline{\theta}_i+\Delta_i$, it follows that  
\begin{align}\label{closef0}
{\delta}_{i}^\top u+\underline{{\eta}}_{i} \le f_{0,i}(x,u)\le {\delta}_{i}^\top u+\overline{{\eta}}_{i} 
\end{align}
holds for any $x\in[\bm x], u\in[\bm u]$.

%See Appendix~\ref{sssec:taylor} for details.

\subsection{Control-affine Enclosure of \texorpdfstring{$f_{\mathrm{NN}}$}{fNN}}
Suppose that the FFNN residual dynamics $f_{\mathrm{NN}}$  is written as $f_{\mathrm{NN}}=(f_{{\mathrm{NN}},1},f_{{\mathrm{NN}},2},\dots,f_{{\mathrm{NN}},n})^\top$. We derive  a control-affine enclosure for each $f_{{\mathrm{NN}},i}$ by leveraging the affine bounds from linear relaxation--based perturbation analysis (LiRPA) method~\cite{Zhang2018CROWN}. Specifically, 
by \cite[Theorem 3.2]{Zhang2018CROWN}, for any  $i\in\{1,2,\dots,n\}$, there exist $\underline{{\gamma}}_{u,i}\in\R^{m}$, $\overline{{\gamma}}_{u,i}\in\R^{m}$, $\underline{{\gamma}}_{x,i}\in\R^{n}$, $\overline{{\gamma}}_{x,i}\in\R^{n}$,  $\overline{\beta}_i\in\R$, and $\underline{\beta}_i\in\R$, 
%computed from the network weights, biases, and activation relaxations, 
such that 
\begin{align}\label{closureNN}
\underline{{\gamma}}_{u,i}^\top u+\underline{{\gamma}}_{x,i}^\top x+\underline{\beta}_i\!\le \!f_{\mathrm{NN},i}(x,u)\!\le\! \overline{{\gamma}}_{u,i}^\top u+\overline{{\gamma}}_{x,i}^\top x+\overline{\beta}_i
\end{align}
holds for any $x\in[\bm x], u\in[\bm u]$. 
By letting $\overline{\zeta}_i = \max_{x\in[\bm x]}\overline{{\gamma}}_{x,i}^\top x+\overline{\beta}_i$ and $\underline{\zeta}_i = \min_{x\in[\bm x]}\underline{{\gamma}}_{x,i}^\top x+\underline{\beta}_i$, equation \eqref{closureNN} yields
\begin{align}\label{closeNN2}
\underline{{\gamma}}_{u,i}^\top u+\underline{\zeta}_i\le f_{\mathrm{NN},i}(x,u)\le \overline{{\gamma}}_{u,i}^\top u+\overline{\zeta}_i.
\end{align}

Combining the control-affine enclosure for $f_{0,i}$ in \eqref{closef0} with that for $f_{\mathrm{NN},i}$ in \eqref{closeNN2}, we obtain a control-affine enclosure $[f]_{[\bm x]}(u)$ for $f$ as expressed in \eqref{closuref}-\eqref{eq:combined_img}, where $\underline{a}_i={\delta}_{i}+\underline{{\gamma}}_{u,i}$, $\overline{a}_i={\delta}_{i}+\overline{{\gamma}}_{u,i}$, $\underline{c}_i=\underline{\eta}_i+\underline{\zeta}_i$, and $\overline{c}_i=\overline{\eta}_i+\overline{\zeta}_i$.

\section{Finite Input Search Reduction via Hyperplane Arrangement}
\label{ssec:critical_ctrl}

In this section, we introduce hyperplane arrangements in the control space based on the control-affine enclosure developed in Section \ref{ssec:affine}, and show how they enable the reduction of the infinite control search required for verifying the CIS condition to a finite set of evaluations.  

Suppose that $\cX$ and $\cU$ are boxes given by $\cX = \bigtimes_{j=1}^{n} [\underline{x}_j,\, \overline{x}_j]$ and $\cU = \bigtimes_{k=1}^{m} [\underline{u}_k,\, \overline{u}_k]$, respectively, and $\mathbb{S}$ is a subpaving of $\cX$ given as $\mathbb{S} = \{[\bm b_1], \ldots, [\bm b_N]\}$. To verify the relaxed CIS condition introduced in Section \ref{sec:suffCIS}, we aim to determine, for each $[\bm b] \in \mathbb{S}$, whether there exists a control input $u \in \cU$ such that $[f]_{[\bm b]}(u) \subseteq \mathbb{S}$, where $[f]_{[\bm b]}(u)$ denotes the control-affine 
enclosure~\eqref{eq:combined_img} constructed over $[\bm b]$. 
%As established in Section~\ref{sec:suffCIS}, fix a subpaving $\mathbb{S} = \{[\bm b_1], \ldots, [\bm b_N]\}$ of $\cX$.
% Our goal is to test whether $\mathbb{S}$ satisfies the relaxed controlled invariance condition, i.e., for each box $[\bm b] \in \mathbb{S}$, to find $u \in \cU$ such that $[f]_{[\bm b]}(u) \subseteq \mathbb{S}$, where $[f]_{[\bm b]}(u)$ is the control-affine 
% enclosure~\eqref{eq:combined_img} constructed over $[\bm b]$. 
%This test can be regarded as a simpler case: we want to determine whether a box ($[f]_{[\bm b]}(u)$) is contained in a union of boxes ($\mathbb{S}$). 
As $u$ varies continuously in the control space $\cU$, the enclosure box $[f]_{[\bm b]}(u)$ also varies continuously in the state space $\cX$. Therefore, the set containment relation above can change only when a lower or upper bound of the box $[f]_{[\bm b]}(u)$ crosses the coordinate of the boundary of boxes in $\mathbb{S}$. Thus, the containment relation changes in a discrete manner. 

Mathematically, for each coordinate $j \in \{1,\ldots,n\}$, let $G_j := \{g_j^{(q)}\}_{q=1}^{N_j}$ denote the sorted set of all $j$-th coordinates of vertices of boxes in $\mathbb{S}$, i.e., $g_j^{(1)} < g_j^{(2)} < \cdots < g_j^{(N_j)}$.
Letting $p = (p_1, \ldots, p_n)$ as the indexes of the boxes induced by $G_j$, where $ p_j\in\{1, \ldots, N_j - 1\}$, we define the \emph{grid cell} as $  C^\circ_p \;:=\; 
  \bigtimes_{j=1}^{n}
  \bigl(g_j^{(p_j)},\; g_j^{(p_j+1)}\bigr).$
%We illustrate these concepts through a concrete example in 2-D state space.

%even though $[f]_{[\bm b]}(u)$ changes continuously, the truth value of the containment relation does not change continuously with $u$; rather, it changes in a discrete manner.

\begin{figure}[t]
  \centering
  \includegraphics[width=\linewidth]{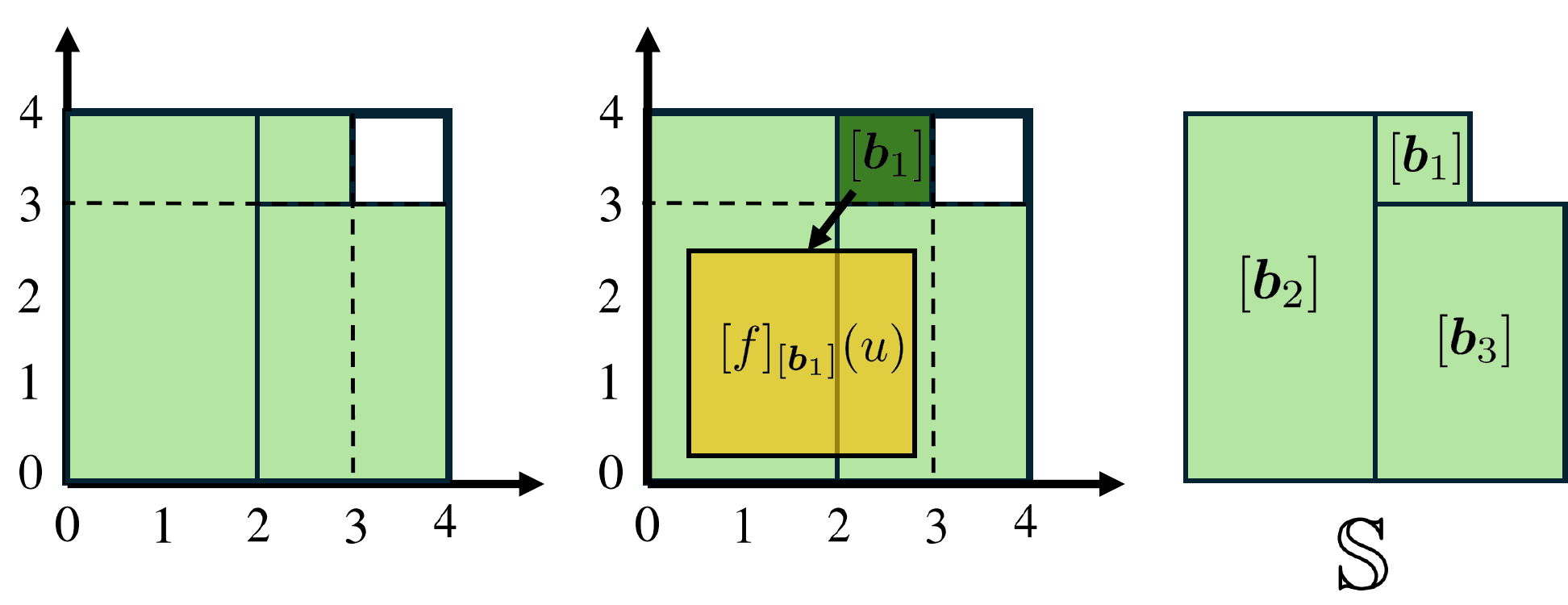}
  \caption{Illustration of Example~\ref{ex:2d_state} with $n=2$ and $\cX=[0,4]\times[0,4]$. 
  Left: the subpaving $\mathbb{S}$ (light green) with the grid lines (including dashed lines) induced by 
  $G_1 = \{0, 2, 3, 4\}$ and $G_2 = \{0, 3, 4\}$.
  Middle: the box $[\bm b_1] \in \mathbb{S}$ (dark green) and the control-affine enclosure box $[f]_{[\bm b_1]}(u)$ for a particular $u$ (yellow).
  Right: the subpaving $\mathbb{S} = \{[\bm b_1], [\bm b_2], [\bm b_3]\}$.}
  \label{fig:example_2d}
\end{figure}

% We begin by introducing some relevant concepts in the state space. 
% For each coordinate $j \in \{1,\ldots,n\}$, let $G_j := \{g_j^{(q)}\}_{q=1}^{N_j}$ denote the sorted set of all $j$-th coordinates of vertices of boxes in $\mathbb{S}$, i.e., $g_j^{(1)} < g_j^{(2)} < \cdots < g_j^{(N_j)}$.
% Note that $g_j^{(1)} = \underline{x}_j$ and $g_j^{(N_j)} = \overline{x}_j$.
% For each $p = (p_1, \ldots, p_n) \in \bigtimes_{j=1}^{n}\{1, \ldots, N_j - 1\}$, we define the \emph{grid cell} $  C^\circ_p \;:=\; 
%   \bigtimes_{j=1}^{n}
%   \bigl(g_j^{(p_j)},\; g_j^{(p_j+1)}\bigr).$
% We illustrate these concepts through a concrete example in 2-D state space.

\begin{example}\label{ex:2d_state}
Consider $n = 2$ and $\cX = [0,4] \times [0,4]$, as illustrated in Figure~\ref{fig:example_2d}.
Suppose $\mathbb{S} = \{[\bm b_1], [\bm b_2], [\bm b_3]\}$ is a subpaving of $\cX$ whose boxes have boundaries with 
coordinates $G_1 = \{0, 2, 3, 4\}$ and $G_2 = \{0, 3, 4\}$, so that $N_1 = 4$ and $N_2 = 3$.
The grid cells are the $(N_1 - 1) \times (N_2 - 1) = 3 \times 2 = 6$ open rectangles $C^\circ_p = (g_1^{(p_1)}, g_1^{(p_1+1)}) \times (g_2^{(p_2)}, g_2^{(p_2+1)})$; for instance, $C^\circ_{(1,1)} = (0,2) \times (0,3)$ and 
$C^\circ_{(2,2)} = (2,3) \times (3,4)$.
In Figure~\ref{fig:example_2d} (left), the grid cells contained in $\mathbb{S}$ are those covered by light green, 
while the remaining grid cell $C^\circ_{(3,2)}$ (white) is disjoint from $\mathbb{S}$.
Now choose the box $[\bm b_1] \in \mathbb{S}$ to be tested (the dark green box in Figure~\ref{fig:example_2d}, middle).
For a particular control $u$, the enclosure $[f]_{[\bm b_1]}(u)$ is the yellow rectangle.
In this configuration, $[f]_{[\bm b_1]}(u)$ intersects two grid cells $C^\circ_{(1,1)}$ and $C^\circ_{(2,1)}$, 
both of which lie inside $\mathbb{S}$, so the containment $[f]_{[\bm b_1]}(u) \subseteq \mathbb{S}$ holds.
If $u$ were to change so that $[f]_{[\bm b_1]}(u)$ shifts to the right and begins to intersect the white grid cell, or moves to extend beyond $\cX$, the containment would fail.
\end{example}

\begin{figure}[t]
  \centering
  \includegraphics[width=\linewidth]{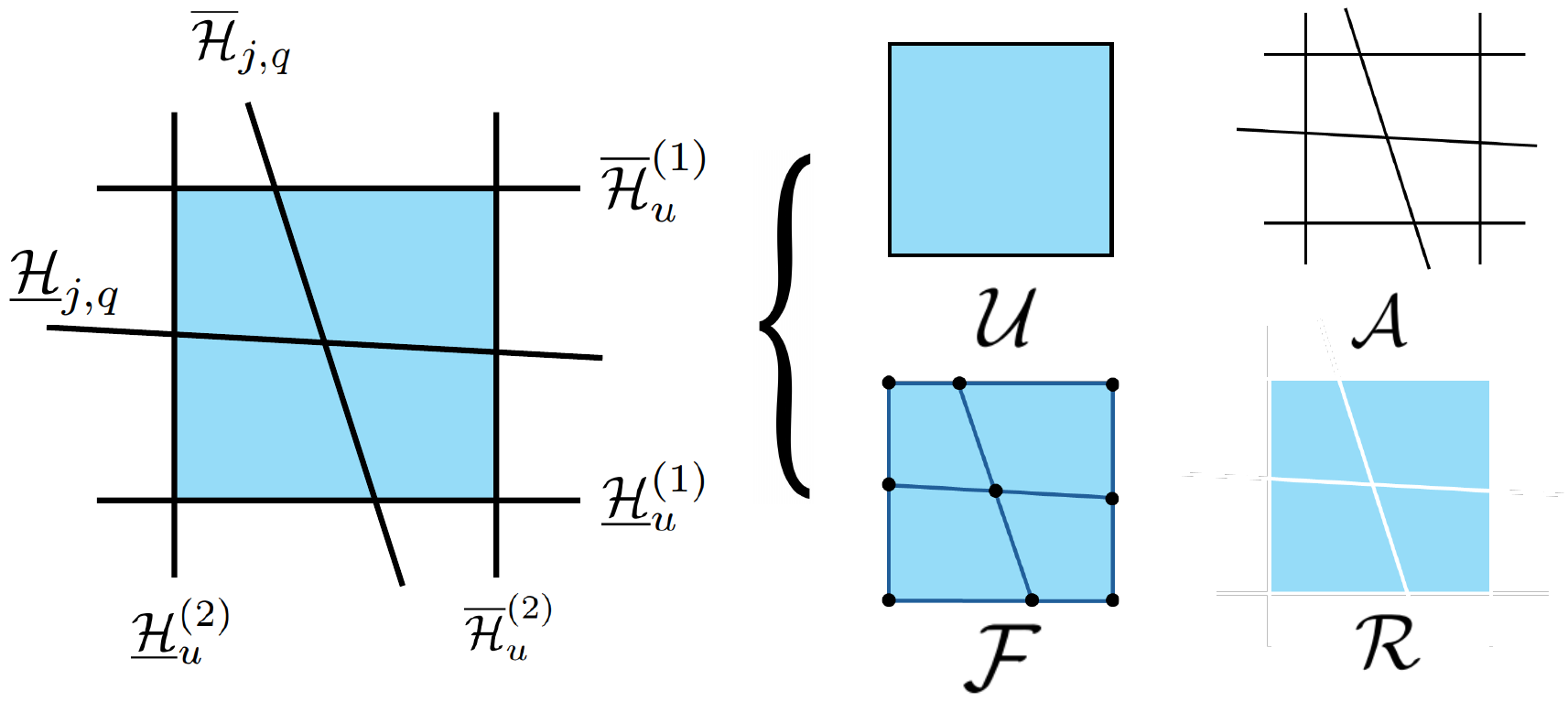}
  \caption{Illustration of Example~\ref{ex:2d_control} with $m=2$. Left: the combined view. Right: four components shown separately. $\mathcal{U}$: the control box. $\mathcal{A}$: the hyperplane arrangement $\mathcal{A}([\bm b])$, consisting of the hyperplanes $\underline{\mathcal{H}}_{j,q}$ and $\overline{\mathcal{H}}_{j,q}$ from~\eqref{eq:critical_eq} together with the boundary hyperplanes $\underline{\mathcal{H}}_u^{(k)}$ and $\overline{\mathcal{H}}_u^{(k)}$ of $\cU$. $\mathcal{F}$: the faces of $\mathcal{A}([\bm b])$ within $\cU$, including vertices, open line segments, and open polygons. $\mathcal{R}$: the regions, i.e., the full-dimensional faces.}
  \label{fig:example_control}
\end{figure}

Now we introduce hyperplane arrangements in the control space $\cU$. 
From the preceding discussion, it is clear that the containment $[f]_{[\bm b]}(u) \subseteq \mathbb{S}$ can change only when a lower or upper bound of $[f]_{[\bm b]}(u)$ crosses a grid coordinate $g_j^{(q)} \in G_j$.
Since each bound depends affinely on $u$, each such crossing event corresponds to an affine hyperplane in the control space $\cU$.
More precisely, for each $j \in \{1,\ldots,n\}$ and each $q \in \{1,\ldots,N_j\}$, the crossing of the lower and upper bounds of $[f]_{[\bm b]}(u)$ with $g_j^{(q)}$ defines the two types of affine hyperplanes
\begingroup
\setlength{\abovedisplayskip}{3pt}  % default is ~12pt
\setlength{\belowdisplayskip}{3pt}
\begin{equation}\label{eq:critical_eq}
\begin{aligned}
  \underline{\mathcal{H}}_{j,q}
  &:=
  \bigl\{u \in \R^m :
    \underline{a}_j^{\top} u + \underline{c}_j 
    = g_j^{(q)}\bigr\},\\
  \overline{\mathcal{H}}_{j,q}
  &:=
  \bigl\{u \in \R^m :
    \overline{a}_j^{\top} u + \overline{c}_j 
    = g_j^{(q)}\bigr\},
\end{aligned}
\end{equation}
\endgroup
where $\underline{a}_j, \overline{a}_j, \underline{c}_j, \overline{c}_j$ are the enclosure coefficients determined by $[\bm b]$ and $\cU$ via~\eqref{eq:combined_img}.
To account for the boundary of $\cU$, we include the $2m$ additional boundary hyperplanes $\underline{\mathcal{H}}_u^{(k)} :=\{u \in \R^m : u_k = \underline{u}_k\}$ and $\overline{\mathcal{H}}_u^{(k)} := 
\{u \in \R^m : u_k = \overline{u}_k\}$ for $k = 1, \ldots, m$. Collecting all these hyperplanes yields a hyperplane arrangement in $\R^m$, denoted by $\mathcal{A}([\bm b])$. Recall from Section~\ref{sec:pre} that each face of $\mathcal{A}([\bm b])$ is a maximal connected set of points sharing the same sign pattern, and that the full-dimensional faces are called regions. Clearly, the arrangement $\mathcal{A}([\bm b])$ partitions $\cU$ into finitely many faces. We denote $\mathcal{F}$ as the set of all faces of $\mathcal{A}([\bm b])$ contained in $\cU$, and $\mathcal{R} = \{R_1, \ldots, R_N\} \subset \mathcal{F}$ the set of regions.

%We again illustrate these concepts through a concrete example.

\begin{example}\label{ex:2d_control}
Consider a two-dimensional control space with $\cU = [\underline{u}_1, \overline{u}_1] \times [\underline{u}_2, \overline{u}_2] \subset \R^2$ as illustrated in Figure~\ref{fig:example_control}.
For a fixed box $[\bm b] \in \mathbb{S}$, suppose that, as $u$ varies in $\cU$, the lower bound of $[f]_{[\bm b]}(u)$ in coordinate $j$ attains the grid coordinate $g_j^{(q)}$, i.e., $\underline{a}_j^\top u + \underline{c}_j = g_j^{(q)}$.
Since $m = 2$, this equation defines a line $\underline{\mathcal{H}}_{j,q}$ in $\R^2$.
Similarly, the upper bound meeting $g_j^{(q)}$ gives another line $\overline{\mathcal{H}}_{j,q}$.
Repeating this for all coordinates $j = 1, \ldots, n$ and all grid indices $q = 1, \ldots, N_j$ produces the full collection of hyperplanes.
However, for illustration, we only focus on the two hyperplanes $\underline{\mathcal{H}}_{j,q}$ and $\overline{\mathcal{H}}_{j,q}$.
Together with the four boundary lines $\underline{\mathcal{H}}_u^{(k)}$, $\overline{\mathcal{H}}_u^{(k)}$ for $k=1,2$, these lines form the arrangement $\mathcal{A}([\bm b])$.
The faces $\mathcal{F}$ of $\mathcal{A}([\bm b])$ within $\cU$, bottom left) 
consist of 9 vertices, 12 open line segments, and 4 open convex polygonal regions $\mathcal{R} = \{R_1, R_2, R_3, R_4\}$.
\end{example}

%We have now established how the geometry of $[f]_{[\bm b]}(u)$ and $\mathbb{S}$ in the state space induces a hyperplane arrangement $\mathcal{A}([\bm b])$ in the control space. 
%The following proposition establishes how the geometry of $[f]_{[\bm b]}(u)$ and $\mathbb{S}$ in the state space induces a hyperplane arrangement $\mathcal{A}([\bm b])$ in the control space. 
%formalizes the key results of this structure.

\begin{proposition}\label{prop:finite_bp}
Let $[\bm b] \in \mathbb{S}$ and assume that $[f]_{[\bm b]}(u)$ is a non-degenerate box for every $u \in \cU$.
Define the covering indicator $\varphi(u) := \mathbb{I}([f]_{[\bm b]}(u) \subseteq \mathbb{S}) \in \{0,1\}$.
The following two statements hold:
\begin{enumerate}[label=(\roman*), nosep]
  \item\label{item:face_const}
  For each face $F \in \mathcal{F}$, $\varphi(u)$ is invariant for any $u \in F$.

  \item\label{item:closure_prop}
  If $\varphi(u) = 1$ on a region $R_i \in \mathcal{R}$, then $\varphi(u) = 1$ on every face contained in $\overline{R}_i$.
\end{enumerate}
\end{proposition}

\begin{proof}
\ref{item:face_const}
% Since $|\mathcal{A}([\bm b])| \leq 2\sum_{j=1}^{n} N_j + 2m < \infty$, the arrangement $\mathcal{A}([\bm b])$ induces finitely many faces~\cite{orlik2013arrangements}.
By definition, all points on a face $F$ share the same sign pattern.
For the proof, we augment $G_j$ by setting $g_j^{(0)}:=-\infty$ and $g_j^{(N_j+1)}:=\infty$, and extend the definition of a grid cell to $C_p^\circ=\bigtimes_{j=1}^{n}(g_j^{(p_j)},g_j^{(p_j+1)})$ with $p_j\in\{0,\ldots,N_j\}$.
For any $u \in \cU$, let $\mathcal{C}(u)$ denote the set of grid cells intersected by $\mathrm{int}([f]_{[\bm b]}(u))$. 
Since $G_j$ contains all $j$-th coordinates of the boundaries of boxes in $\mathbb{S}$, each grid cell is either contained in $\mathbb{S}$ or disjoint from $\mathbb{S}$.

The non-degeneracy assumption ensures that $[f]_{[\bm b]}(u)$ has nonempty interior and satisfies $[f]_{[\bm b]}(u) = \overline{\mathrm{int}([f]_{[\bm b]}(u))}$. 
Since $\mathbb{S}$ is closed, $[f]_{[\bm b]}(u)\subseteq\mathbb{S}$ if and
only if $\mathrm{int}([f]_{[\bm b]}(u))\subseteq\mathbb{S}$. If
$\mathrm{int}([f]_{[\bm b]}(u))\not\subseteq\mathbb{S}$, then the open set
$\mathrm{int}([f]_{[\bm b]}(u))\cap(\R^n\setminus\mathbb{S})$ intersects at lease one grid cell disjoint from $\mathbb{S}$.
Therefore, $\varphi(u)=1$ if and only if every cell in $\mathcal{C}(u)$ is contained in $\mathbb{S}$. Equivalently, $\varphi(u)=0$ if and only if some cell in $\mathcal{C}(u)$ is disjoint from $\mathbb{S}$.

Let $u,u'\in F$. Since $u$ and $u'$ share the same sign pattern, it remains
to show that $\mathcal{C}(u)$ is completely determined by $s(u)$. 
The open box $\mathrm{int}([f]_{[\bm b]}(u))=\bigtimes_{j=1}^{n}(\underline{a}_j^{\top}u+\underline{c}_j, \overline{a}_j^{\top}u+\overline{c}_j)$ intersects a grid cell $C_p^\circ=\bigtimes_{j=1}^{n}(g_j^{(p_j)},g_j^{(p_j+1)})$ if and only if, for every $j$, $\underline{a}_j^{\top}u+\underline{c}_j<g_j^{(p_j+1)}$ and
$\overline{a}_j^{\top}u+\overline{c}_j>g_j^{(p_j)}$.
For $1\leq p_j\leq N_j-1$, the truth values of these two inequalities are
determined by the sign-pattern entries associated with $\underline{\mathcal{H}}_{j,p_j+1}$ and $\overline{\mathcal{H}}_{j,p_j}$, respectively, as defined in \eqref{eq:critical_eq}.
When $p_j=0$, the second inequality holds trivially because $g_j^{(0)}=-\infty$, while the first inequality is determined by the sign-pattern entry associated with $\underline{\mathcal{H}}_{j,1}$.
Similarly, when $p_j=N_j$, the first inequality holds trivially because
$g_j^{(N_j+1)}=\infty$, while the second inequality is determined by the
sign-pattern entry associated with $\overline{\mathcal{H}}_{j,N_j}$.
Therefore, whether $\mathrm{int}([f]_{[\bm b]}(u))$ intersects any given grid
cell is completely determined by $s(u)$, and hence so is $\mathcal{C}(u)$.
Since $s(u)=s(u')$, we have $\mathcal{C}(u)=\mathcal{C}(u')$, which implies
$\varphi(u)=\varphi(u')$. Thus, $\varphi$ is constant on $F$.

\ref{item:closure_prop}
We first prove that the set $U([\bm b])=\{u\in\cU:[f]_{[\bm b]}(u)\subseteq\mathbb{S}\}=\{u\in\cU:\varphi(u)=1\}$ is closed in $\cU$.
The complement of $U([\bm b])$ consists of those $u\in\cU$ for which $\mathrm{int}([f]_{[\bm b]}(u))$ intersects at least one grid cell $C_p^\circ$ disjoint from $\mathbb{S}$.
For any fixed such cell, this intersection is equivalent to the strict
inequalities $\underline{a}_j^{\top}u+\underline{c}_j<g_j^{(p_j+1)}$ and
$\overline{a}_j^{\top}u+\overline{c}_j>g_j^{(p_j)}$ for every $j$, where an
inequality involving $g_j^{(0)}=-\infty$ or $g_j^{(N_j+1)}=\infty$ is automatically satisfied.
Hence, the set of controls for which a fixed grid cell is intersected is open
in $\cU$. Since there are finitely many grid cells, the complement of $U([\bm b])$ is a finite union of open sets, so $U([\bm b])$ is closed in
$\cU$.

If $\varphi=1$ on a region $R_i$, then $R_i\subseteq U([\bm b])$.
Since $U([\bm b])$ is closed in $\cU$, we have
$\overline{R}_i\subseteq U([\bm b])$, and hence $\varphi=1$ for all points of
$\overline{R}_i$. Equivalently, $\varphi=1$ on every face contained in
$\overline{R}_i$.
% Thus, every face $F\subseteq\overline{R}_i$ satisfies $\varphi=1$ on~$F$.
% \medskip
% \noindent\emph{Proof of~\ref{item:global_equiv}.}
% The faces in $\mathcal{F}$ partition~$[u_s]$. Therefore, $\exists\,u\in[u_s]:\varphi(u)=1$ if and only if $\varphi=1$ on some face $F\in\mathcal{F}$. By part~\ref{item:face_const}, this holds if and only if
% $\varphi(u_F^*)=1$ for the representative of that face.
\end{proof}
% \HZ{Add some sentences to explain why Proposition \ref{prop:finite_bp} converts ``infinite search'' to finite.}

% Proposition~\ref{prop:finite_bp} shows that the infinite control search required for verifying the CIS condition can be reduced to a finite set of evaluations: 

Proposition~\ref{prop:finite_bp} shows that the covering indicator $\varphi(u)$ is invariant on each face, so it suffices to evaluate $\varphi$ at one representative point per face. Since the number of faces is finite, the search for an admissible control input in the CIS condition reduces to a finite set of candidates. Furthermore,  from the proof of Proposition~\ref{prop:finite_bp} it follows  that $U([\bm b])$ is a finite union of convex polytopes, since each face closure is a convex polytope and $\mathcal{A}([\bm b])$ has finitely many faces. Hence, $U([\bm b])$ provides a uniform inner approximation of the regulation map over the box~$[\bm b]$.

\section{Algorithms for MCIS Approximations}\label{ssec:complete}

%This section combines the results of the preceding sections into a complete algorithm for Problem~\ref{prob:main}.
In this section, we present Algorithms~\ref{alg:outer}-\ref{alg:escape} for constructing inner and outer approximations of the MCIS, along with their theoretical guarantees. In practice, the tightness of the control-affine enclosure depends on the width of $[\bm u]$, as the Taylor remainder bound of $f_0$ contains quadratic terms in $w([\bm u])$ while the LiRPA relaxation gap of $f_{\mathrm{NN}}$ grows with $w([\bm u])$. To mitigate the potential enclosure conservatism, we partition $\cU$ into control slices 
$$\cU = \bigcup_{i=1}^{N_u} [\bm u_s^{(i)}]$$ and compute a separate control-affine enclosure on each slice. In the following algorithms, we instantiate the generic subpaving $\mathbb{S}$ in Sections~\ref{sec:suffCIS}--\ref{ssec:critical_ctrl} as $\cC_k$, the candidate CIS subpaving at the $k$-th iteration.

Algorithm~\ref{alg:outer} takes as input the state constraint set $\cX$, the control constraint set $\cU$, a resolution parameter $\varepsilon > 0$, and the number of control slices $N_u$, and returns a subpaving $\cC$ that is certified to be a CIS, or $\emptyset$ if no CIS is found at the given resolution.
The algorithm performs a fixed-point iteration: starting from $\cC_0 = \{\cX\}$ (line~1), at each iteration $k$ it calls Algorithm~\ref{alg:inner} to compute $\cC_{k+1} \subseteq \cC_k$ (line~3) by retaining only those boxes that pass the relaxed CIS test. The iteration terminates when either $\cC_{k+1} = \emptyset$ (lines~4--5) or $\cC_{k+1} = \cC_k$ (lines~6--7).

Algorithm~\ref{alg:inner} processes each box $[\bm b] \in \cC_k$ from a queue $\mathcal{Q}$ (lines~1,~3--4). For each control slice $[\bm u_s^{(i)}]$ (line~2,~6), the algorithm computes the control-affine enclosure coefficients via~\eqref{eq:combined_img} (line~7), constructs the hyperplane arrangement $\mathcal{A}([\bm b])$ (line~8), and enumerates its regions (line~9).
For the enumeration of polyhedral regions induced by a hyperplane arrangement, standard region-enumeration methods can be employed, such as the Incremental Enumeration method~\cite{rada2018new}.
For each region $R$, one representative $u^* \in R$ is selected and tested for $[f]_{[\bm b]}(u^*) \subseteq \cC_k$ (lines~10--11). 
For computational efficiency, Algorithm~\ref{alg:inner} only enumerates and tests the full-dimensional regions of the hyperplane arrangement. By Proposition~\ref{prop:finite_bp}\ref{item:closure_prop}, once a region is certified feasible, all lower-dimensional faces contained in its closure are also feasible and therefore need not be tested separately.
If a feasible $u^*$ is found on any slice, $[\bm b]$ is added to $\cC_{k+1}$ (line~13-14); otherwise, if $w([\bm b]) > \varepsilon$, $[\bm b]$ is bisected into two boxes $[\bm b]^{(1)},[\bm b]^{(2)}$ satisfying $[\bm b]^{(1)}\cup[\bm b]^{(2)}=[\bm b]$ with disjoint interiors, which are then returned to the queue (lines~15--16). Boxes with $w([\bm b]) \leq \varepsilon$ that fail the test are discarded.

\begin{algorithm}[b]
\caption{Fixed-point iteration for CIS}\label{alg:outer}
\KwIn{$\mathcal{X},\;\mathcal{U},\;\varepsilon,\;N_u$}
\KwOut{$\cC$ as a CIS, or $\emptyset$}
$\cC_0 \leftarrow \{\mathcal{X}\}$; \quad $k\leftarrow 0$; \quad $\mathrm{converged}\leftarrow\mathrm{false}$\;
\While{$\neg\,\mathrm{converged}$}{
  $\cC_{k+1} \leftarrow \textsc{InPre}(\cC_k,\;\varepsilon,\;\mathcal{U},\;N_u)$\tcp*{Alg.~\ref{alg:inner}}
  
  \If{$\cC_{k+1}=\emptyset$}{
    \Return{$\emptyset$}\tcp*{No CIS at resolution $\varepsilon$}
  }
  \If{$\cC_{k+1}=\cC_k$}{
    $\mathrm{converged}\leftarrow\mathrm{true}$\;
  }
  $k\leftarrow k+1$\;
}
\Return{$\cC_k$}
\end{algorithm}

\begin{algorithm}[t]
\caption{\textsc{InPre}: Inner approximation of $\mathrm{Pre}(\cC_k)$}\label{alg:inner}
\KwIn{$\cC_k,\;\varepsilon,\;\cU,\;N_u$}
\KwOut{$\cC_{k+1}$ with $\cC_{k+1}\subseteq\cC_k$}
$\mathcal{Q} \leftarrow \cC_k$;\quad $\cC_{k+1} \leftarrow \emptyset$\;
Partition $\cU$ into $N_u$ slices $[\bm u_s^{(1)}],\ldots,[\bm u_s^{(N_u)}]$\;
\While{$\mathcal{Q} \neq \emptyset$}{
  $[\bm b] \leftarrow \mathrm{Pop}(\mathcal{Q})$\;
  % $\Phi_{\cC_k}([\bm b]) \leftarrow \emptyset$\;
  $\mathrm{feasible} \leftarrow \mathrm{false}$\;
  \For{$i=1,\ldots,N_u$}{
    $\bigl\{\underline{a}_j,\,\overline{a}_j,\,\underline{c}_j,\,\overline{c}_j\bigr\}_{j=1}^{n} \leftarrow [f]_{[\bm b]}$ over $[\bm u_s^{(i)}]$\tcp*{\eqref{eq:combined_img}}
    $\mathcal{A}([\bm b])\leftarrow\bigcup_{j=1}^{n}\bigcup_{q=1}^{N_j}\bigl\{\overline{\mathcal{H}}_{j,q},\underline{\mathcal{H}}_{j,q}\bigr\}\cup\bigcup_{k=1}^{m}
  \bigl\{\{u:u_k=\underline{u}_{s,k}^{(i)}\},\{u:u_k=\overline{u}_{s,k}^{(i)}\}\bigr\}$\;
    \ForEach{region\/ $R$ of\/ $\mathcal{A}([\bm b])$ within\/ $[\bm u_s^{(i)}]$}{
      Choose one $u^*\in R$\;
      \If{$\varphi(u^*)=1$}{
        $\mathrm{feasible} \leftarrow \mathrm{true}$\;
      }
    }
  }
  \uIf{$\mathrm{feasible}$}{
    $\cC_{k+1}\leftarrow\cC_{k+1}\cup\{[\bm b]\}$\;
  }
  \ElseIf{$w([\bm b])>\varepsilon$}{
    bisect $[\bm b]\to[\bm b]^{(1)},[\bm b]^{(2)}$;\quad
    $\mathcal{Q}\leftarrow\mathcal{Q}\cup\{[\bm b]^{(1)},[\bm b]^{(2)}\}$\;
  }
}
\Return{$\cC_{k+1}$}
\end{algorithm}

% Algorithm~\ref{alg:inner} requires checking whether $[f](u^*)$ is contained in the union $\cC_k$. A naive scan over all boxes in $\cC_k$ per query would make this step the computational bottleneck. To alleviate this, we accelerate the verification using a spatial hashing method. 

% We construct a uniform-grid spatial index exactly once per iteration, requiring an $O(n(\cC_k))$ setup time, where $n(\cC_k)$ means the number of boxes in $\cC_k$. This makes each paving box registered into the hash buckets of its overlapping smallest grid units. For each image box $[f](u^*)$, we identify the minimal grid cells it covers. By querying the corresponding hash buckets of these cells (which store the indices of the subpaving boxes overlapping them), we can determine the candidate boxes that potentially intersect with $[f](u^*)$. This step drastically reduces the search space from scanning all boxes to evaluating only the potentially intersecting ones. Then, we iteratively subtract the volume of each candidate, allowing for an early return as soon as the residual becomes empty. 

% Because the hash index limits the candidate size to a small constant, this spatial hashing method reduces the per-query coverage check to an expected $O(1)$ complexity.

Algorithm~\ref{alg:escape} computes a subpaving $\cO$ of $\cX$ satisfying $\cO\subseteq\cX\setminus\cS^*$, which certifies that every box in $\cO$ lies outside the MCIS. Together with $\cC$ from Algorithm~\ref{alg:outer}, the remaining states that belong neither to $\cC$ nor to $\cO$ are collected in a subpaving $\mathbb{U}$. This yields a paving of $\cX$ satisfying $\cX=\mathbb{C} \cup \cO \cup \mathbb{U}$ where $\mathbb{C} \subseteq \cS^*$ and $\cO \subseteq \mathcal{X}\setminus \cS^*$. Since the procedure is adopted from ~\cite{monnet2016computing}, we briefly outline the idea below. 
For a given state box $[\bm b]$ and control slice $[\bm u_s^{(i)}]$, the inclusion function $[f]$ evaluated on the joint box $[\bm b] \times [\bm u_s^{(i)}]$ satisfies
\begingroup
\setlength{\abovedisplayskip}{3pt}  % default is ~12pt
\setlength{\belowdisplayskip}{3pt}
\begin{equation}\label{eq:reach_slice}
  [f]([\bm b],\,[\bm u_s^{(i)}]) \supseteq 
  \{f(x,u) : x \in [\bm b],\, u \in [\bm u_s^{(i)}]\}.
\end{equation}
\endgroup
This can be computed using standard interval enclosure techniques~\cite{moore2009introduction,chabert2009contractor}.
A box $[\bm b]$ is then classified as outside the MCIS if, for every control slice $i \in \{1,\ldots,N_u\}$, the over-approximated image $[f]([\bm b], [\bm u_s^{(i)}])$ either lies outside $\cX$ or is contained in the already-classified outside region $\cO_k$ at iteration $k$. The algorithm terminates when the outside subpaving stabilizes, i.e., $\cO_{k+1} = \cO_k$.
% \TY{Done.}\HZ{$\Box_{u\in[\bm u_s^{(i)}]} [f_i](u)$ may mean $\Box [f_i]([\bm u_s^{(i)}])$?}

\begin{algorithm}[t]
\caption{Determining States Outside MCIS}
\label{alg:escape}
\KwIn{$\cX,\;\cU,\;\varepsilon,\;N_u$}
\KwOut{$\cO$ with $\cO\subseteq\cX\setminus\cS^\ast$}
$\cO_0 \leftarrow \emptyset$;\quad $k \leftarrow 0$;\quad
  $\mathrm{converged}\leftarrow\mathrm{false}$\;
Partition $\cU$ into $N_u$ slices
  $[\bm u_s^{(1)}],\ldots,[\bm u_s^{(N_u)}]$\;
\While{$\neg\,\mathrm{converged}$}{
  $\mathcal{Q} \leftarrow \cX$;\quad
    $\cO_{k+1} \leftarrow \cO_k$\;
  \While{$\mathcal{Q} \neq \emptyset$}{
    $[\bm b] \leftarrow \mathrm{Pop}(\mathcal{Q})$\;
    $\mathrm{outside} \leftarrow \mathrm{true}$\;
    \For{$i=1,\ldots,N_u$}{
      Compute $[f]([\bm b],\,[\bm u_s^{(i)}])$%
        \tcp*{\eqref{eq:reach_slice}}
      \If{$[f]([\bm b],\,[\bm u_s^{(i)}])
            \cap \cX \not\subseteq \cO_k$}{
        $\mathrm{outside} \leftarrow \mathrm{false}$\;
      }
    }
    \uIf{$\mathrm{outside}$}{
      $\cO_{k+1} \leftarrow \cO_{k+1}
        \cup \{[\bm b]\}$\;
    }
    \ElseIf{$w([\bm b]) > \varepsilon$}{
      bisect $[\bm b]\to[\bm b]^{(1)},[\bm b]^{(2)}$;\quad
      $\mathcal{Q} \leftarrow \mathcal{Q}
        \cup \{[\bm b]^{(1)},[\bm b]^{(2)}\}$\;
    }
  }
  \If{$\cO_{k+1} = \cO_k$}{
    $\mathrm{converged}\leftarrow\mathrm{true}$\;
  }
  $k \leftarrow k+1$\;
}
\Return{$\cO_k$}
\end{algorithm}

The following theorems show the correctness and termination guarantees of the proposed algorithms.

\begin{theorem}\label{thm:correctness}
For any $\varepsilon > 0$ and $N_u \geq 1$, the subpaving $\cC$ returned by Algorithm~\ref{alg:outer} is a CIS for system~\eqref{eq:system}.
\end{theorem}

\begin{proof}
The control-affine enclosure is a valid inclusion function for the full dynamics $f = f_0 + f_{\mathrm{NN}}$, hence for every box $[\bm b] \in \cC$, there exists $u \in \cU$ such that $f([\bm b], u) \subseteq [f]_{[\bm b]}(u) \subseteq \cC$. Note that $u$ may vary for different choices of $[\bm b]$.
Now let $x$ be any point in $\cC$. Then $x$ belongs to some box $[\bm b] \in \cC$, and therefore $f(x, u) \in f([\bm b], u) \subseteq \cC$.
Since this holds for every $x \in \cC$, we conclude that $\cC$ is a CIS.
\end{proof}

\begin{theorem}\label{thm:outer_sound}
For any $\varepsilon > 0$ and $N_u \geq 1$, the subpaving $\cO$ returned by Algorithm~\ref{alg:escape} satisfies $\cO\subseteq\cX\setminus\cS^*$.
\end{theorem}

\begin{proof}
We show by induction on $k$ that
$\cO_k\cap\cS^*=\emptyset$.

\emph{Base case.}
$\cO_0=\emptyset$, so the claim holds trivially.

\emph{Inductive step.}
Assume $\cO_k\cap\cS^*=\emptyset$.
Let $[\bm b]$ be a box added to $\cO_{k+1}$, so that $[f]([\bm b],[\bm u_s^{(i)}])\cap\cX\subseteq\cO_k$ for all $i\in\{1,\ldots,N_u\}$.
Since $\cU = \bigcup_{i=1}^{N_u} [\bm u_s^{(i)}]$, for any $x\in[\bm b]$ and any $u\in\cU$, there exists $i$ such that $(x,u)\in[\bm b]\times[\bm u_s^{(i)}]$, so by~\eqref{eq:reach_slice}, $f(x,u)\in[f]([\bm b],[\bm u_s^{(i)}])$.
Since $[f]([\bm b],[\bm u_s^{(i)}])\cap\cX\subseteq\cO_k$, either $f(x,u)\notin\cX$ or $f(x,u)\in\cO_k$.
In the first case, $f(x,u)\notin\cS^*$ since $\cS^*\subseteq\cX$; in the second case, $f(x,u)\notin\cS^*$ by the inductive hypothesis $\cO_k\cap\cS^*=\emptyset$.
Since this holds for all $u\in\cU$, $x\notin\mathrm{Pre}(\cS^*)$.
By $\cS^*\subseteq\mathrm{Pre}(\cS^*)$, it follows that $x\notin\cS^*$.
Thus, $[\bm b]\cap\cS^*=\emptyset$ and $\cO_{k+1}\cap\cS^*=\emptyset$.
\end{proof}

\begin{theorem}\label{thm:termination}
For any $\varepsilon > 0$ and $N_u\geq 1$, Algorithms~\ref{alg:outer},~\ref{alg:inner}, and~\ref{alg:escape} terminate in a finite number of steps.
\end{theorem}

\begin{proof}
In Algorithm~\ref{alg:inner}, each box popped from the queue is either added to $\cC_{k+1}$, bisected, or discarded. Bisection reduces the width of the box by half. After at most $\lceil\log_2(w(\cX)/\varepsilon)\rceil$ bisections along any single dimension, the box width falls below $\varepsilon$ and the box is discarded, so Algorithm~\ref{alg:inner} terminates.

% For Algorithm~\ref{alg:outer}, the invariance test can only retain or remove boxes, so $\cC_{k+1}\subseteq\cC_k$ at every iteration.
% Since the bisection tree of $\cX$ at resolution $\varepsilon$ has finite leaf nodes, the collection of all subpavings at this resolution is finite.
% The sequence $\cC_0\supseteq\cC_1\supseteq\cdots$ is a descending sequence taking values in a finite collection. Since a finite set admits no infinite strictly descending sequence, this sequence can only strictly decrease finite times and hence must stabilize.
% % That is, there exists $N$ such that $\psu{\cC_n}=\psu{\cC_N}$ for all $n\geq N$.

% For Algorithm~\ref{alg:escape}, the escape test can only retain or add boxes, so $\cO_k\subseteq\cO_{k+1}$ at every iteration.
% Since the bisection tree of $\cX$ at resolution $\varepsilon$ has finitely many leaf nodes, the collection of all subpavings at this resolution is finite.
% The sequence $\cO_0\subseteq\cO_1\subseteq\cdots$ is an ascending sequence taking values in a finite collection. Since a finite set admits no infinite strictly ascending sequence, this sequence can only strictly increase finite times and hence must stabilize.
For Algorithms~\ref{alg:outer} and~\ref{alg:escape}, the bisection tree of $\cX$ at resolution $\varepsilon$ has finitely many leaves, so the collection of all subpavings at this resolution is finite. Since $\cC_{k+1} \subseteq \cC_k$ and $\cO_k \subseteq \cO_{k+1}$ at every iteration, both sequences are monotone in a finite collection and hence must terminate.
\end{proof}

% \begin{remark}\label{rem:regulation}
%     Since each face closure is a convex polytope and $\mathcal{A}([\bm b])$ has finitely many faces, $U([\bm b])$ is a finite union of convex polytopes. Hence $U([\bm b])$ provides a uniform inner approximation of the regulation map over the box~$[\bm b]$.
% \end{remark}

%\begin{remark}\label{rem:parallel}
Since the computations for different boxes and different control representatives are independent, they are naturally parallelizable. In our implementation, state box and control slice pairs are collected into batches, and the matrix multiplications in the LiRPA bound propagation are executed concurrently on the GPU.

\section{Simulation Examples}
\label{sec:experiments}

We present two simulation examples to demonstrate the effectiveness of the proposed method. All experiments are conducted on a desktop equipped with an AMD Ryzen 7 7700X 8-Core CPU, 32\,GB RAM, and an NVIDIA GeForce RTX~5070 Ti GPU with 16\,GB VRAM.
% The lane-keeping example follows the setup of~\cite{li2025control}, using the same bicycle
% dynamics~\eqref{eq:bicycle_num} and the same $3$--$8$--$4$--$2$ MLP with $\text{ReLU}$ activations.
% The frozen-turbulence navigation example uses a $2$--$256$--$256$--$256$--$2$ MLP with $\tanh$ activations and is included to demonstrate that the proposed method
% applies directly to smooth activations and large networks without modification.

%\subsection{Lane-keeping}\label{ssec:lanekeeping}
% \begin{example}\label{ssec:lanekeeping} {\bf (Lane Keeping)} 
\subsection{Lane Keeping}\label{ssec:lanekeeping}

Consider the discrete-time bicycle model in \cite{li2025control}:
\begingroup
\setlength{\abovedisplayskip}{3pt}  % default is ~12pt
\setlength{\belowdisplayskip}{3pt}
\begin{equation}
\label{eq:bicycle_num}
\begin{aligned}
      y_{k+1} &= y_k + v\,dt\sin\theta_k,\\
    \theta_{k+1} &= \theta_k + {v\,dt} \tfrac{\tan u_k}{l_1},
\end{aligned}
\end{equation}
\endgroup
where $y_k\in\R$ is the lateral displacement of the vehicle, $\theta_k\in\R$ is the yaw angle, $u_k\in\R$ is the steering angle, and the parameters are chosen as $l_1=5$\,m (vehicle length), $l_2=2$\,m (vehicle width),  $v=6$\,m/s (vehicle speed), and $w=3.5$\,m (lane width). 
The state is defined as $x_k=(y_k,\;y_k+l_1\theta_k)^\top\in\R^2$, with scalar control $u_k\in\R$. We consider the NNCS \eqref{eq:system} where $f_0\equiv\bm{0}$ and $f_{\mathrm{NN}}$ is a $3$-layer, ReLU-activated FFNN with two hidden layers of sizes $n_1=8,n_2=4$, trained to imitate the bicycle model in \eqref{eq:bicycle_num}. Specifically, the network is of the form $x_{k+1}=f_{\mathrm{NN}}(x_k,u_k)$ and is trained on $2\times 10^5$ training samples and validated on $5\times 10^4$ samples for $200$ epochs, achieving a final validation MSE of $5.06\times 10^{-7}$. We consider the state constraint set $\cX = [-0.75,\;0.75]\times [-0.75,\;0.75]$ and two different control constraint sets: $\cU=[-5^\circ,\;5^\circ]$ or $[-10^\circ,\;10^\circ]$.

%\XX{what do you mean $\cU$ belongs to a set?}

% The state constraint set and control constraint sets are $\cX = [-0.75,\;0.75]^2,\;\cU \in \bigl\{[-5^\circ,\;5^\circ],\;[-10^\circ,\;10^\circ]\bigr\}$.
% $3$--$8$--$4$--$2$ neurons, taking $(x_k,u_k)$ as input and $x_{k+1}$ as output, serves as $f_{\mathrm{NN}}$ in~\eqref{eq:system} with $f_0\equiv\bm{0}$.
%input $(x_k,u_k)\in\R^3$, output $x_{k+1}\in\R^2$, and 

We fix $N_u = 3$ and vary the resolution $\varepsilon = \frac{w-l_2}{K}$ with $K \in \{32,\,64,\,128,\,256,\,512,\,1024\}$ in both Algorithm~\ref{alg:outer} and~\ref{alg:escape}, 
% \begin{equation}\label{eq:lk_resolutions}
%   \varepsilon = \frac{w-l_2}{K},
%   \qquad
%   K \in \{32,\,64,\,128,\,256,\,512,\,1024\},
% \end{equation}
yielding increasingly fine pavings. Since the control is a scalar, each hyperplane~\eqref{eq:critical_eq} reduces to a point in $\R$, and the regions of the hyperplane arrangement are open intervals whose midpoints could serve as the representatives in Proposition~\ref{prop:finite_bp}.
% \XX{which algorithm we run? } 
Figure~\ref{fig:lk_paving} shows the pavings $\cC,\cO,\mathbb{U}$ as results of Algorithm~\ref{alg:outer} and~\ref{alg:escape}, for varying $K$ and two different $\cU$. It can be observed that the size of $\mathbb{U}$, which is the set of states with  undetermined classification, decreases as $K$ increases, i.e, as the paving becomes finer with higher resolution. Table~\ref{tab:lk_runtime} reports the wall-clock runtimes of Algorithm~\ref{alg:outer} for computing the CIS $\cC$. 
%The runtime of Algorithm~\ref{alg:escape} for computing  $\cO$ is not reported, as computing an outer approximation of the MCIS is secondary to computing the CIS itself.
% \XX{time of running which algorithm? we should tie it to the specific algorithms/theorems in the simulation results.}
%\end{example}

% \begin{figure*}[t]
%   \centering
%   \includegraphics[width=\textwidth]{figures/example_1.pdf}
%   \caption{Simulation results for the lane keeping example in Section \ref{ssec:lanekeeping} with varying $K$ and two different control sets $\cU$. The subpaving $\cC$ (green boxes) represents a CIS as an inner approximation of the MCIS; $\cO$ (gray boxes) contains states outside the MCIS; and  $\mathbb{U}$ (yellow boxes) contains states   with  undetermined classification. It can be observed that the size of $\mathbb{U}$ decreases as $K$ increases, i.e, as the paving becomes finer with higher resolution. }
%   \label{fig:lk_paving}
% \end{figure*}

\begin{figure*}[t]
  \centering

  \begin{subfigure}[t]{\textwidth}
    \centering
    \includegraphics[width=\textwidth]{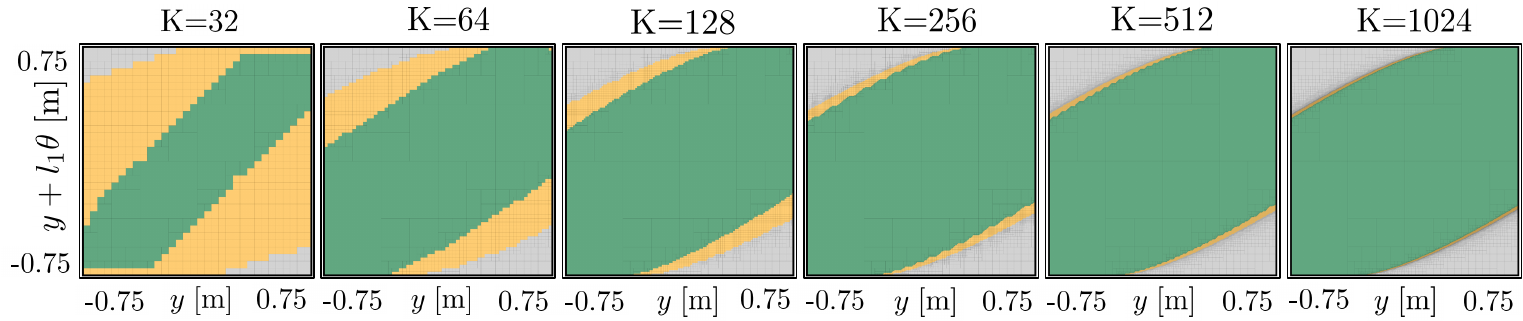}
    \caption{$\cU=[-5^\circ,\,5^\circ]$.}
    \label{fig:lk_paving_u5}
  \end{subfigure}

  \begin{subfigure}[t]{\textwidth}
    \centering
    \includegraphics[width=\textwidth]{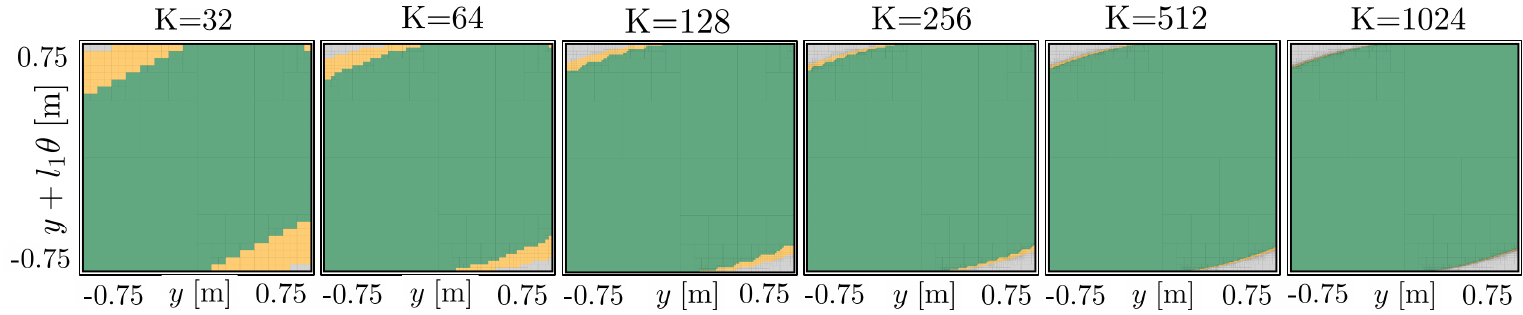}
    \caption{$\cU=[-10^\circ,\,10^\circ]$.}
    \label{fig:lk_paving_u10}
  \end{subfigure}

  \caption{Simulation results for the lane keeping example in Section \ref{ssec:lanekeeping} with varying $K$ and two different control sets $\cU$. The subpaving $\cC$ (green boxes) represents a CIS as an inner approximation of the MCIS; $\cO$ (gray boxes) contains states outside the MCIS; and  $\mathbb{U}$ (yellow boxes) contains states   with  undetermined classification. It can be observed that the size of $\mathbb{U}$ decreases as $K$ increases, i.e, as the paving becomes finer with higher resolution.}
  \label{fig:lk_paving}
\end{figure*}

\begin{table}[b]
  \centering
  \caption{Runtime (in sec) for computing the CIS $\cC$ in the lane keeping example.}
  \label{tab:lk_runtime}
  \small
  \renewcommand{\arraystretch}{1.0}
  \setlength{\tabcolsep}{6pt}
  \begin{tabular}{ccccccc}
    \toprule
    $\cU \backslash K$ & 32 & 64 & 128 & 256 & 512 & 1024 \\
    \midrule
    $[-5^\circ,5^\circ]$   & 0.35 & 0.56 & 0.91 & 1.14 & 1.49 & 3.81 \\
    $[-10^\circ,10^\circ]$ & 0.22 & 0.34 & 0.31 & 0.36 & 0.51 & 0.86 \\
    \bottomrule
  \end{tabular}
\end{table}

%\subsection{Flow Navigation in Frozen Synthetic Turbulence}\label{ssec:exp2}

% \begin{example}\label{ssec:exp2} {\bf (Flow Navigation in Frozen Synthetic Turbulence)} 
\subsection{Flow Navigation in Frozen Synthetic Turbulence}\label{ssec:exp2}
Consider a point agent moving in a two-dimensional frozen incompressible velocity field constructed via 
the random Fourier mode synthesis of~\cite{kraichnan1970diffusion}, with stream function $\psi(x,y) = \sum_{j=1}^{N} A_j \sin(k_{x,j} x + k_{y,j} y + \varphi_j)$ and velocity $\bm{V} = (\partial_y \psi,\, -\partial_x \psi)^\top$. We use $N = 40$ modes with wavevectors $(k_{x,j}, k_{y,j})$ sampled from $\{\pm 1, \ldots, \pm 8\}^2$, random phases $\varphi_j \sim \mathrm{Uniform}(0, 2\pi)$, and amplitudes $A_j = 8 / \|(k_{x,j}, k_{y,j})\|^{2.5}$. We train a $4$-layer, $\tanh$-activated FFNN, $\hat{\bm{V}} = (\hat V_x, \hat V_y)^\top \colon \R^2 \to \R^2$,  with three hidden layers of sizes $n_1 = n_2 = n_3 = 256$ to approximate the velocity field $\bm{V}$. The network is trained on $8 \times 10^5$ samples and validated on $8 \times 10^4$ samples for $500$ epochs, achieving a final validation MSE of $3.48 \times 10^{-4}$. Figure~\ref{fig:vel_field} shows the learned speed map and streamlines. 

\begin{figure}[t]
  \centering
  \includegraphics[width=\linewidth]{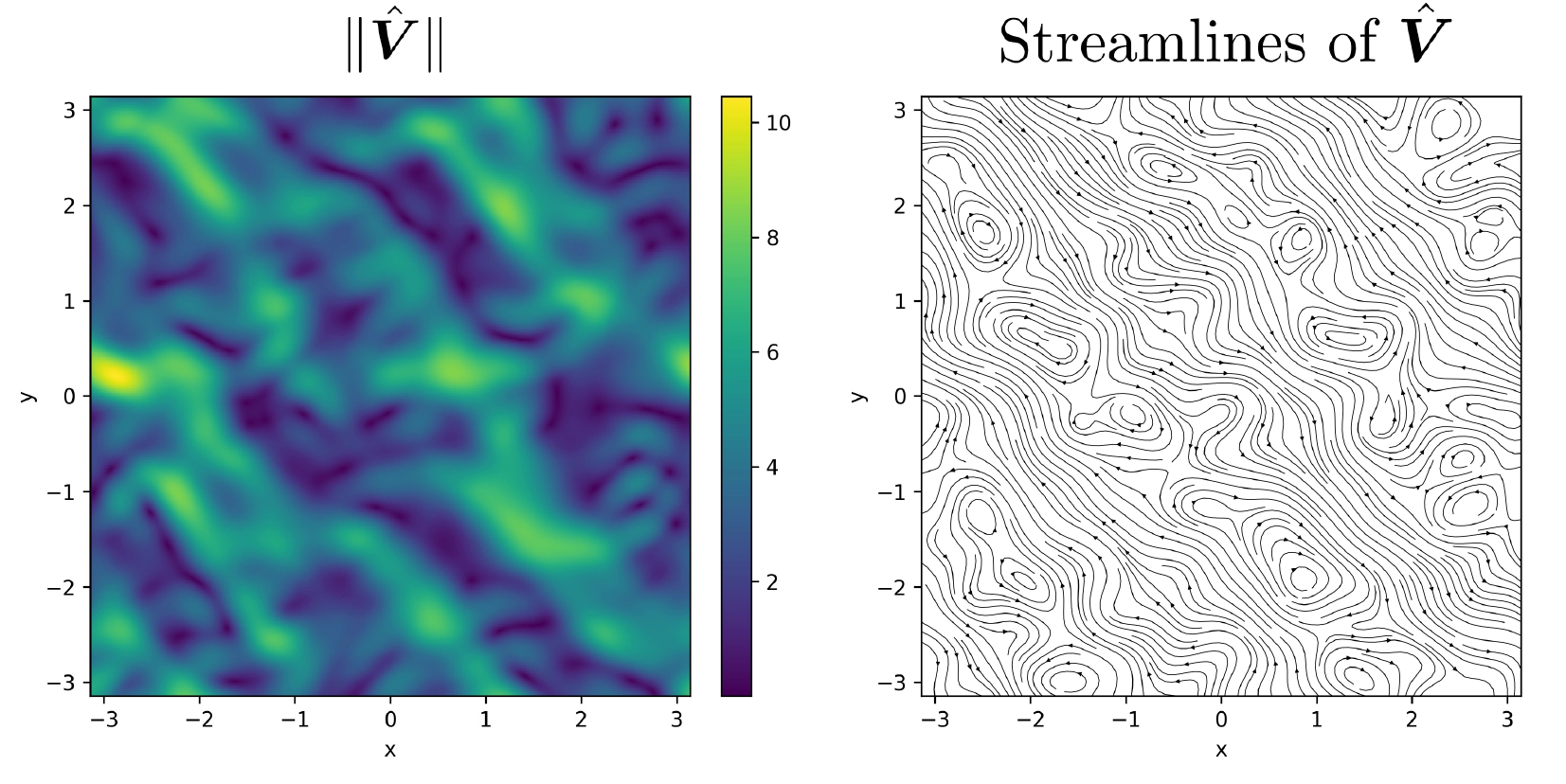}
  \vskip -2mm
  \caption{Learned velocity field $\hat{\bm{V}}$ produced by the FFNN in the flow navigation example.
  Left: speed map $\|\hat{\bm{V}}(x)\|$.
  Right: streamlines of $\hat{\bm{V}}$. }
  \label{fig:vel_field}
\end{figure}

\begin{figure}[t]
  \centering
  \includegraphics[width=\linewidth]{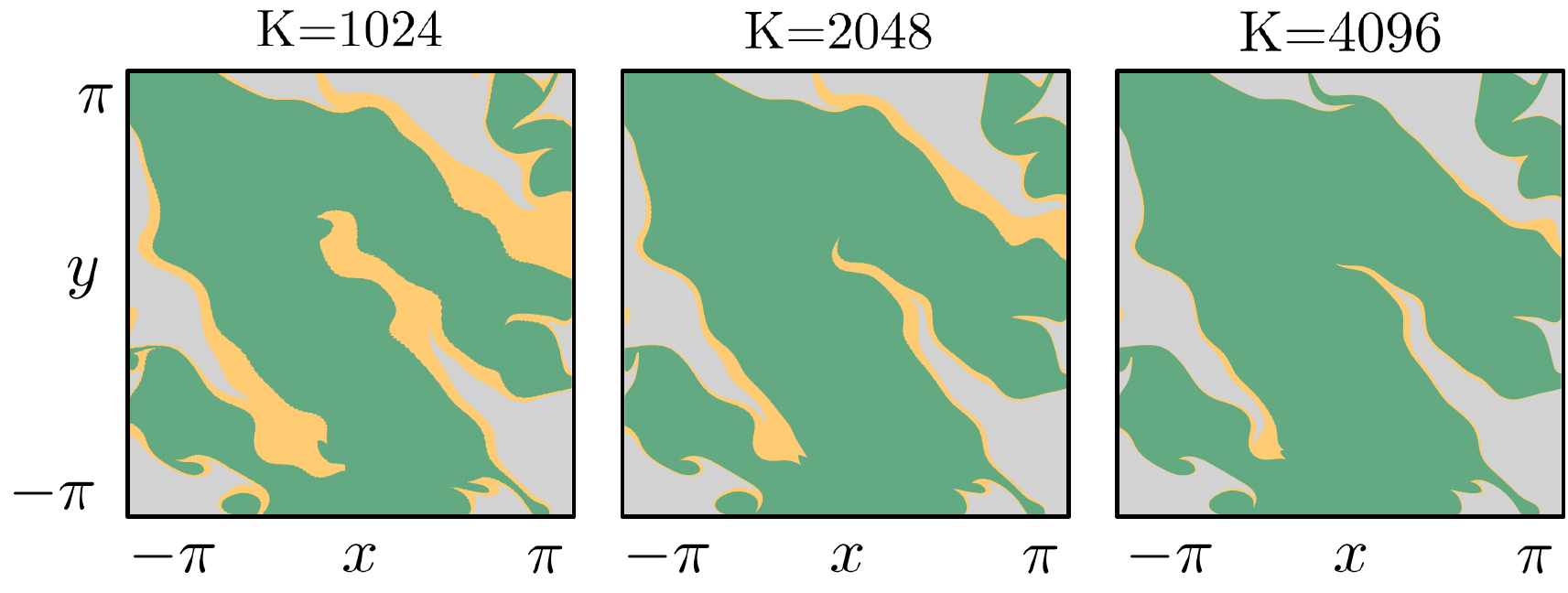}
  \vskip -2mm
  \caption{Simulation results for the flow navigation example in Section \ref{ssec:exp2} with varying $K$. The subpaving $\cC$ (green boxes) represents a CIS as an inner approximation of the MCIS; $\cO$ (gray boxes) contains states outside the MCIS; and  $\mathbb{U}$ (yellow boxes) contains states   with  undetermined classification. 
  %Grid lines are omitted for visual clarity. 
  }
  %\caption{Three-class pavings for the frozen-turbulence navigation example. Green boxes belong to $\cC$, gray boxes belong to $\cO$, and yellow boxes belong to $\mathbb{U}$. Grid lines are omitted for visual clarity.}
  \label{fig:flow_paving}
\end{figure}

% \XX{revise this part similar to Example 1}
% A $2$--$256$--$256$--$256$--$2$ $\tanh$ MLP $\hat{\bm{V}}=(\hat V_x,\hat V_y)^\top$ is trained to regress $\bm{V}$ using the MSE loss on $8\times 10^5$ training samples
% and validated on $8\times 10^4$ samples for $500$ epochs, achieving a final validation MSE of $3.48\times10^{-4}$.
% Figure~\ref{fig:vel_field} shows the learned speed map and streamlines.
% A $4$-layer, $\tanh$-activated FFNN with three hidden layers of sizes $n_1=n_2=n_3=256$ is trained to regress the velocity field $\bm{V}$, yielding a learned network $\hat{\bm{V}}=(\hat V_x,\hat V_y)^\top$. Specifically, the network is trained on $8\times10^5$ training samples and validated on $8\times10^4$ samples for $500$ epochs, achieving a final validation MSE of $3.48\times10^{-4}$. Figure~\ref{fig:vel_field} shows the learned speed map and streamlines.
Now consider the discrete-time flow navigation model:
\begingroup
\setlength{\abovedisplayskip}{3pt}  % default is ~12pt
\setlength{\belowdisplayskip}{3pt}
\begin{equation}\label{eq:flow_dyn_num}
\begin{aligned}
  x_{k+1} &= x_k + dt\bigl(\hat V_x(x_k,y_k) + u_{x,k}\bigr),\\
  y_{k+1} &= y_k + dt\bigl(\hat V_y(x_k,y_k) + u_{y,k}\bigr),
\end{aligned}
\end{equation}
\endgroup
where $x_k$ and $y_k$ denote the agent position coordinates, $u_{x,k}$ and $u_{y,k}$ are the control inputs, and $dt = 0.02$\,s is the sampling period.
The state is defined as $(x_k, y_k)^\top \in \R^2$ with control $u_k = (u_{x,k}, u_{y,k})^\top \in \R^2$.
We consider the NNCS~\eqref{eq:system} where $f_0(x,u) = x + dt\,u$ and $f_{\mathrm{NN}}(x_k) = dt\,\hat{\bm{V}}(x_k, y_k)$. The state and control constraint sets are chosen as  $\cX = [-\pi,\,\pi] \times [-\pi,\,\pi]$ and $\cU = [-0.5,\,0.5] \times [-0.5,\,0.5]$. We fix $N_u = 1$, i.e., no control slicing, and vary the resolution $\varepsilon = \frac{2\pi}{K}$ with 
$K \in \{1024,\, 2048,\, 4096\}$ in both Algorithms~\ref{alg:outer} and~\ref{alg:escape}.
Since the control is two-dimensional, each hyperplane~\eqref{eq:critical_eq} is a line in $\R^2$, and the regions of the arrangement are open convex polygons. Figure~\ref{fig:flow_paving} shows the pavings $\cC, \cO, \mathbb{U}$ as results of Algorithms~\ref{alg:outer} and~\ref{alg:escape} for varying $K$. As in the lane keeping example, the size of $\mathbb{U}$ decreases as $K$ increases, i.e., as the paving becomes finer with higher resolution. The wall-clock runtimes of Algorithm~\ref{alg:outer} for computing the CIS $\cC$ are $79.5$\,s, $245.1$\,s, and $2056.3$\,s for $K = 1024$, $2048$, and $4096$, respectively.

% Now consider a discrete-time flow navigation model:
% %based on $\hat{\bm{V}}$,
% \begin{equation}\label{eq:flow_dyn_num}
% \begin{aligned}
%   x_{k+1} &= x_k + dt\bigl(\hat V_x(x_k,y_k) + u_{x,k}\bigr),\\
%   y_{k+1} &= y_k + dt\bigl(\hat V_y(x_k,y_k) + u_{y,k}\bigr),
% \end{aligned}
% \end{equation}
% where $x_k$ and $y_k$ denote the agent position coordinates, and $u_{x,k}$ and $u_{y,k}$ are the control inputs.
% %The sampling period is $dt=0.02$\,s.
% \XX{Revise similar to Example 1.}
% The state is defined as $x_k=(x_k,y_k)^\top\in\R^2$ with control $u_k=(u_{x,k},u_{y,k})^\top\in\R^2$, and the learned velocity network $\hat{\bm{V}}=(\hat V_x,\hat V_y)^\top$ serves as $f_{\mathrm{NN}}$ in~\eqref{eq:system}, while $f_0(x,u)=x+dt\,u$.
% The state and control constraint sets are $\cX = [-\pi,\;\pi]\times [-\pi,\;\pi]$ and  
% $\cU = [-0.5,\;0.5]\times [-0.5,\;0.5].$

% We then fix $N_u=1$, i.e., no control slicing, and test three resolutions $\varepsilon = \frac{2\pi}{K}$ with $ K\in\{1024,\;2048,\;4096\}.$

% \XX{Revise similar to Example 1. Tie it to the algorithms. Explain the figures similar to Example 1.}
% Figure~\ref{fig:flow_paving} displays the resulting three-class pavings.
% The wall-clock runtimes for computing the CIS $\cC$ are $79.5$\,s, $245.1$\,s, and $2056.3$\,s for $K=1024$, $2048$, and $4096$, respectively.

%\end{example}

\section{Conclusion} \label{sec:concl}
In this paper, we proposed an optimization-free approach for computing inner and outer approximations of the maximal controlled invariant set for NNCSs. By leveraging the control-affine inclusion function of NNCS dynamics and hyperplane arrangements, the infinite control search for verifying CISs is reduced to a finite set of evaluations. 
Verification-guided algorithms are developed with inherent parallelizability. 
% Future work includes improving approximation tightness and quantifying the approximation error.
% This paper proposed an optimization-free method for computing certified inner and outer approximations of the MCIS for discrete-time systems with neural-network components.
% A control-affine enclosure and the induced finite partition of the control space reduce the infinite search over admissible controls to finitely many representatives, making the method naturally GPU-parallelizable.
Simulation results on lane keeping and frozen-turbulence navigation demonstrated the effectiveness of the proposed method in the computation of three-class pavings and certified CIS.
Future work includes the extension to robust controlled invariance under uncertainty and the quantification of approximation errors.

\bibliographystyle{IEEEtran}
% \printbibliography
\bibliography{ref}

\end{document}